\documentclass[12pt]{amsart}      
\usepackage{txfonts}
\usepackage{amssymb}
\usepackage{eucal}
\usepackage{graphicx}
\usepackage{amsmath}
\usepackage{amsfonts}
\usepackage{amscd}
\usepackage[all]{xy}           
\usepackage[active]{srcltx} 

\usepackage{amsfonts,latexsym}
\usepackage{xspace}
\usepackage{epsfig}
\usepackage{float}
\usepackage{color}
\usepackage{fancybox}
\usepackage{colordvi}
\usepackage{multicol}
\usepackage{colordvi}
\usepackage[hypertex]{hyperref} 

\newtheorem{theorem}{Theorem}[section]
\newtheorem{proposition}[theorem]{Proposition}
\newtheorem{definition}[theorem]{Definition}
\newtheorem{lemma}[theorem]{Lemma}
\newtheorem{corollary}[theorem]{Corollary}
\newtheorem{prop-def}{Proposition-Definition}[section]
\newtheorem{coro-def}{Corollary-Definition}[section]

\newtheorem{remark}[theorem]{Remark}

\newtheorem{example}{Example}[section]

\newcommand{\nc}{\newcommand}
\nc{\tred}[1]{\textcolor{red}{#1}}
\nc{\tblue}[1]{\textcolor{blue}{#1}}
\nc{\tgreen}[1]{\textcolor{green}{#1}}
\nc{\tpurple}[1]{\textcolor{purple}{#1}}
\nc{\btred}[1]{\textcolor{red}{\bf #1}}
\nc{\btblue}[1]{\textcolor{blue}{\bf #1}}
\nc{\btgreen}[1]{\textcolor{green}{\bf #1}}
\nc{\btpurple}[1]{\textcolor{purple}{\bf #1}}

\renewcommand{\Bbb}{\mathbb}

\newcommand{\efootnote}[1]{}

\renewcommand{\textbf}[1]{}

\newcommand{\delete}[1]{}
\nc{\dfootnote}[1]{{}}          
\nc{\ffootnote}[1]{\dfootnote{#1}}
\nc{\mfootnote}[1]{\footnote{#1}} 
\nc{\ofootnote}[1]{\footnote{\tiny Older version: #1}}

\nc{\mlabel}[1]{\label{#1}}  
\nc{\mcite}[1]{\cite{#1}}  
\nc{\mref}[1]{\ref{#1}}  
\nc{\mbibitem}[1]{\bibitem{#1}} 

\delete{
\nc{\mlabel}[1]{\label{#1}  
{\hfill \hspace{1cm}{\bf{{\ }\hfill(#1)}}}}
\nc{\mcite}[1]{\cite{#1}{{\bf{{\ }(#1)}}}}  
\nc{\mref}[1]{\ref{#1}{{\bf{{\ }(#1)}}}}  
\nc{\mbibitem}[1]{\bibitem[\bf #1]{#1}} 
}

\nc{\mtail}{\leq_t}
\nc{\mhead}{\leq_h}
\nc{\rk}{\mathrm{rk}}
\nc{\mset}[1]{\tilde{#1}}
\nc{\pa}{\frakL}
\nc{\arr}{\rightarrow}
\nc{\lu}[1]{(#1)}
\nc{\mult}{\mrm{mult}}
\nc{\diff}{\mathrm{Der}}
\nc{\indiff}{\mathrm{InDer}}
\nc{\outdiff}{\mathrm{OutDer}}
\nc{\conmat}{connection matrix\xspace}
\nc{\bounmat}{boundary matrix\xspace}
\nc{\pcyc}{\mathfrak c}
\nc{\calpa}{\calp_A}
\nc{\calpal}{\Gamma_{AL}}
\nc{\calpc}{\calp_L}
\nc{\frakDa}{\frakD_1}
\nc{\frakDal}{\frakD_2}
\nc{\frakDc}{\frakD_L}
\nc{\frakDv}{\frakD_V}
\nc{\frakDp}{\frakD_F}
\nc{\frakBa}{\frakB_1}
\nc{\frakBal}{\frakB_2}
\nc{\frakBc}{\frakB_L}
\nc{\frakBv}{\frakB_V}

\nc{\bin}[2]{ (_{\stackrel{\scs{#1}}{\scs{#2}}})}  
\nc{\binc}[2]{ \left (\!\! \begin{array}{c} \scs{#1}\\
    \scs{#2} \end{array}\!\! \right )}  
\nc{\bincc}[2]{  \left ( {\scs{#1} \atop
    \vspace{-1cm}\scs{#2}} \right )}  
\nc{\bs}{\bar{S}}
\nc{\cosum}{\sqsubset}
\nc{\la}{\longrightarrow}
\nc{\rar}{\rightarrow}
\nc{\dar}{\downarrow}
\nc{\dprod}{**}
\nc{\dap}[1]{\downarrow \rlap{$\scriptstyle{#1}$}}
\nc{\md}{\mathrm{dth}}
\nc{\uap}[1]{\uparrow \rlap{$\scriptstyle{#1}$}}
\nc{\defeq}{\stackrel{\rm def}{=}}
\nc{\disp}[1]{\displaystyle{#1}}
\nc{\dotcup}{\ \displaystyle{\bigcup^\bullet}\ }
\nc{\gzeta}{\bar{\zeta}}
\nc{\hcm}{\ \hat{,}\ }
\nc{\hts}{\hat{\otimes}}
\nc{\barot}{{\otimes}}
\nc{\free}[1]{\bar{#1}}
\nc{\uni}[1]{\tilde{#1}}
\nc{\hcirc}{\hat{\circ}}
\nc{\lleft}{[}
\nc{\lright}{]}
\nc{\lc}{\lfloor}
\nc{\rc}{\rfloor}
\nc{\curlyl}{\left \{ \begin{array}{c} {} \\ {} \end{array}
    \right .  \!\!\!\!\!\!\!}
\nc{\curlyr}{ \!\!\!\!\!\!\!
    \left . \begin{array}{c} {} \\ {} \end{array}
    \right \} }
\nc{\longmid}{\left | \begin{array}{c} {} \\ {} \end{array}
    \right . \!\!\!\!\!\!\!}
\nc{\onetree}{\bullet}
\nc{\ora}[1]{\stackrel{#1}{\rar}}
\nc{\ola}[1]{\stackrel{#1}{\la}}
\nc{\ot}{\otimes}
\nc{\mot}{{{\boxtimes\,}}}
\nc{\otm}{\overline{\boxtimes}}
\nc{\sprod}{\bullet}
\nc{\scs}[1]{\scriptstyle{#1}}
\nc{\mrm}[1]{{\rm #1}}
\nc{\margin}[1]{\marginpar{\rm #1}}   
\nc{\dirlim}{\displaystyle{\lim_{\longrightarrow}}\,}
\nc{\invlim}{\displaystyle{\lim_{\longleftarrow}}\,}
\nc{\mvp}{\vspace{0.3cm}}
\nc{\tk}{^{(k)}}
\nc{\tp}{^\prime}
\nc{\ttp}{^{\prime\prime}}
\nc{\svp}{\vspace{2cm}}
\nc{\vp}{\vspace{8cm}}
\nc{\proofbegin}{\noindent{\bf Proof: }}
\nc{\proofend}{$\blacksquare$ \vspace{0.3cm}}
\nc{\modg}[1]{\!<\!\!{#1}\!\!>}
\nc{\intg}[1]{F_C(#1)}
\nc{\lmodg}{\!<\!\!}
\nc{\rmodg}{\!\!>\!}
\nc{\cpi}{\widehat{\Pi}}
\nc{\sha}{{\mbox{\cyr X}}}  
\nc{\shap}{{\mbox{\cyrs X}}} 
\nc{\shpr}{\diamond}    
\nc{\shp}{\ast}
\nc{\shplus}{\shpr^+}
\nc{\shprc}{\shpr_c}    
\nc{\msh}{\ast}
\nc{\zprod}{m_0}
\nc{\oprod}{m_1}
\nc{\vep}{\varepsilon}
\nc{\labs}{\mid\!}
\nc{\rabs}{\!\mid}

\nc{\mmbox}[1]{\mbox{\ #1\ }}
\nc{\fp}{\mrm{FP}} \nc{\rchar}{\mrm{char}} \nc{\End}{\mrm{End}} \nc{\Fil}{\mrm{Fil}}
\nc{\Mor}{Mor\xspace}
\nc{\gmzvs}{gMZV\xspace}
\nc{\gmzv}{gMZV\xspace}
\nc{\mzv}{MZV\xspace}
\nc{\mzvs}{MZVs\xspace}
\nc{\Hom}{\mrm{Hom}} \nc{\id}{\mrm{id}} \nc{\im}{\mrm{im}}
\nc{\incl}{\mrm{incl}} \nc{\map}{\mrm{Map}} \nc{\mchar}{\rm char}
\nc{\nz}{\rm NZ} \nc{\supp}{\mathrm Supp}

\nc{\Alg}{\mathbf{Alg}}
\nc{\Bax}{\mathbf{Bax}}
\nc{\bff}{\mathbf f}
\nc{\bfk}{{\bf k}}
\nc{\bfone}{{\bf 1}}
\nc{\bfx}{\mathbf x}
\nc{\bfy}{\mathbf y}
\nc{\base}[1]{\bfone^{\otimes ({#1}+1)}} 
\nc{\Cat}{\mathbf{Cat}}

\nc{\detail}{\marginpar{\bf More detail}
    \noindent{\bf Need more detail!}
    \svp}
\nc{\Int}{\mathbf{Int}}
\nc{\Mon}{\mathbf{Mon}}
\nc{\rbtm}{{shuffle }}
\nc{\rbto}{{Rota-Baxter }}
\nc{\remarks}{\noindent{\bf Remarks: }}
\nc{\Rings}{\mathbf{Rings}}
\nc{\Sets}{\mathbf{Sets}}

\nc{\BA}{{\Bbb A}} \nc{\CC}{{\Bbb C}} \nc{\DD}{{\Bbb D}}
\nc{\EE}{{\Bbb E}} \nc{\FF}{{\Bbb F}} \nc{\GG}{{\Bbb G}}
\nc{\HH}{{\Bbb H}} \nc{\LL}{{\Bbb L}} \nc{\NN}{{\Bbb N}}
\nc{\KK}{{\Bbb K}} \nc{\QQ}{{\Bbb Q}} \nc{\RR}{{\Bbb R}}
\nc{\TT}{{\Bbb T}} \nc{\VV}{{\Bbb V}} \nc{\ZZ}{{\Bbb Z}}

\nc{\cala}{{\mathcal A}} \nc{\calc}{{\mathcal C}}
\nc{\cald}{{\mathcal D}} \nc{\cale}{{\mathcal E}}
\nc{\calf}{{\mathcal F}} \nc{\calg}{{\mathcal G}}
\nc{\calh}{{\mathcal H}} \nc{\cali}{{\mathcal I}}
\nc{\call}{{\mathcal L}} \nc{\calm}{{\mathcal M}}
\nc{\caln}{{\mathcal N}} \nc{\calo}{{\mathcal O}}
\nc{\calp}{{\mathcal P}} \nc{\calr}{{\mathcal R}}
\nc{\cals}{{\mathcal S}}
\nc{\calt}{{\mathcal T}} \nc{\calw}{{\mathcal W}}
\nc{\calk}{{\mathcal K}} \nc{\calx}{{\mathcal X}}
\nc{\CA}{\mathcal{A}}

\nc{\fraka}{{\mathfrak a}}
\nc{\frakA}{{\mathfrak A}}
\nc{\frakb}{{\mathfrak b}}
\nc{\frakB}{{\mathfrak B}}
\nc{\frakC}{{\mathfrak C}}
\nc{\frakD}{{\mathfrak D}}
\nc{\frakg}{{\mathfrak g}}
\nc{\frakH}{{\mathfrak H}}
\nc{\frakL}{{\mathfrak L}}
\nc{\frakM}{{\mathfrak M}}
\nc{\bfrakM}{\overline{\frakM}}
\nc{\frakm}{{\mathfrak m}}
\nc{\frakP}{{\mathfrak P}}
\nc{\frakN}{{\mathfrak N}}
\nc{\frakp}{{\mathfrak p}}
\nc{\frakR}{{\mathfrak R}}
\nc{\frakS}{{\mathfrak S}}

\font\cyr=wncyr10
\font\cyrs=wncyr7

\begin{document}

\title[Cluster category]
{On a cluster category of type $D_{\infty}$}

%
%

\author[Yichao Yang]{Yichao Yang}
\address{D\'{e}partement de math\'{e}matiques, Universit\'{e} de Sherbrooke, Sherbrooke, Qu\'{e}bec, Canada, J1K 2R1}
\email{yichao.yang@usherbrooke.ca}


\renewcommand{\thefootnote}{\alph{footnote}}
\setcounter{footnote}{-1} \footnote{}
\renewcommand{\thefootnote}{\alph{footnote}}
\setcounter{footnote}{-1} \footnote{\emph{2010 Mathematics Subject
Classification}: 13F60, 16G20, 16G70, 18E30.}

\renewcommand{\thefootnote}{\alph{footnote}}
\setcounter{footnote}{-1} \footnote{\emph{Keywords}: cluster category, cluster-tilting subcategory, Auslander-Reiten theory.}


\begin{abstract}
We study the canonical orbit category of the bounded derived category of finite dimensional representations of the quiver of type $D_{\infty}$. We prove that this orbit category is a cluster category, that is, its cluster-tilting subcategories form a cluster structure as defined in [\ref{BIRS}]. 
\end{abstract}

\maketitle

\setcounter{section}{0}

\section{Introduction}

Cluster algebras, introduced by Fomin and Zelevinsky in [\ref{FZ1}], were to give a combinatorial characterization of dual canonical basis of the quantized enveloping algebra of a quantum group and total positivity for algebraic groups. Now it arises in connection with many branches of mathematics. One of the links is with the representation theory of algebra, where was first discovered in [\ref{MRZ}]. The notion of cluster categories introduced by Buan, Marsh, Reineke, Reiten and Todorov in [\ref{BMRRT}] were to model the main ingredients in the definition of a cluster algebra in a categorical setting. In its original definition, a cluster category is the orbit category of the bounded derived category of finite dimensional representations of a finite acyclic quiver under the composite of the shift functor and the inverse Auslander-Reiten translation. Later on, Buan, Iyama, Reiten and Scott generalized this definition to the notion of cluster structure in a $2$-Calabi-Yau triangulated category in [\ref{BIRS}], where the cluster-tilting objects were replaced by cluster-tilting subcategories. If the cluster-tilting subcategories in a $2$-Calabi-Yau triangulated category form a cluster structure, then this $2$-Calabi-Yau triangulated category is called a \emph{cluster category}. However, it is not true in general. In [\ref{HJ}], Holm and J\o rgensen studied a cluster category of infinite Dynkin type $A_{\infty}$, which is equivalent to the canonical orbit category of the bounded derived category of finite dimensional representations of a quiver of type $A_{\infty}$ with the zigzag orientation, generalising the $A_{n}$ case. In [\ref{LP2}], Liu and Charles extend this work and prove that this canonical orbit category of a quiver of type $A_{\infty}^{\infty}$ without infinite paths is also a cluster category. They also conjecture that this canonical orbit category is a cluster category whenever the original quiver is locally finite without infinite paths. 

The aim of this paper is to extend the above-mentioned works of Holm and J\o rgensen, Liu and Charles to a quiver of type $D_{\infty}$ without infinite paths. Let $Q$ be the quiver of type $D_{\infty}$ with the zigzag orientation. We prove that the canonical orbit category of the bounded derived category of finite dimensional representations of $Q$ is a cluster category and show that this result holds for any quiver of type $D_{\infty}$ without infinite paths, which give a partial answer of the above conjecture and the question mentioned in [\ref{HJ}, (6.1)]. 

This paper is organized as follows. Section $2$ is devoted to recalling the notions of cluster-tilting subcategory, cluster structure and cluster category. We also give the basic information on the category of finite dimensional $k$-linear representations ${\rm rep}(Q)$, the bounded derived category $D^{b}({\rm rep}(Q))$ and the canonical orbit category $\mathcal{C}(Q)$. In particular, we summarizes all the Auslander-Reiten sequences in the Auslander-Reiten quiver $\Gamma_{{\rm rep}(Q)}$ shown in [\ref{RV}]. Section $3$ describes the morphisms in ${\rm rep}(Q)$ and Section $4$ proves our main result; see $(4.12)$.

\section{Preliminaries}

Throughout this paper, $k$ stands for an algebraically closed field. The standard duality for the category of finite dimensional $k$-vector spaces will be denoted by $D$. Throughout this section, $\mathcal{A}$ denotes a $k$-linear triangulated category whose shift functor is denoted by $[1]$. We assume that $\mathcal{A}$ is Hom-finite and Krull-Schmidt. Let $\mathcal{D}$ be a full subcategory of $\mathcal{A}$. Recall that $\mathcal{D}$ is \emph{covariantly finite} in $\mathcal{A}$ provided that every object $X$ of $\mathcal{A}$ admits a \emph{left $\mathcal{D}$-approximation}, that is, a morphism $f: X\rightarrow M$ in $\mathcal{A}$ such that every morphism $g: X\rightarrow N$ with $N\in \mathcal{D}$ factors through $f$; and \emph{contravariantly finite} in $\mathcal{A}$ provided that every object $X$ of $\mathcal{A}$ admits a \emph{right $\mathcal{D}$-approximation}, that is, a morphism $f: M\rightarrow X$ such that every morphism $g: N\rightarrow X$ with $N\in \mathcal{D}$ factors through $f$; and \emph{functorially finite} in $\mathcal{A}$ if it is both covariantly and contravariantly finite in $\mathcal{A}$. Moreover, $\mathcal{A}$ is called \emph{$2$-Calabi-Yau} provided that there are bifunctorial isomorphisms 
${\rm Ext}_{\mathcal{A}}^{1}(X,Y)\cong D {\rm Ext}_{\mathcal{A}}^{1}(Y,X)$
for $X,Y\in \mathcal{A}$; see, for example, [\ref{BMRRT}].

Let $\mathcal{A}$ be a $2$-Calabi-Yau triangulated $k$-category. We recall the following basic definitions from [\ref{BIRS}, \ref{LP2}]. 

\begin{definition}\label{basic definition}

$(1)$~ A full subcategory $\mathcal{T}$ of $\mathcal{A}$ is called \emph{strictly additive} provided that $\mathcal{T}$ is closed under isomorphisms, finite direct sums, and taking summands.

$(2)$~ A strictly additive subcategory $\mathcal{T}$ of $\mathcal{A}$ is called \emph{weakly cluster-tilting} provided that for any $X\in \mathcal{A}$, ${\rm Ext}_{\mathcal{A}}^{1}(\mathcal{T}, X)=0$ if and only if $X\in \mathcal{T}$.

$(3)$~ A strictly additive subcategory $\mathcal{T}$ of $\mathcal{A}$ is called \emph{cluster-tilting} provided that $\mathcal{T}$ is weakly cluster-tilting and functorially finite in $\mathcal{A}$.
\end{definition}

Let $\mathcal{T}$ be a strictly additive subcategory of $\mathcal{A}$. By definition, the quiver of $\mathcal{T}$, denoted by $Q_{\mathcal{T}}$, is the underlying quiver of its Auslander-Reiten quiver. For each indecomposable object $M$ of $\mathcal{T}$, we denote by $\mathcal{T}_{M}$ the full additive subcategory of $\mathcal{T}$ generated by the indecomposable objects not isomorphic to $M$; see [\ref{LP2}].

For convenience, we reformulate the notion of a cluster structure introduced in [\ref{BIRS}].

\begin{definition}\label{cluster structure}
Let $\mathcal{A}$ be a $2$-Calabi-Yau triangulated $k$-category. A non-empty collection $\mathfrak{C}$ of strictly additive subcategories of $\mathcal{A}$ is called a \emph{cluster structure} provided that, for any $\mathcal{T}\in \mathfrak{C}$ and any indecomposable object $M\in \mathcal{T}$, the following conditions are satisfied.

$(1)$~ The quiver of $\mathcal{T}$ has no loops or $2$-cycles.

$(2)$~ There exists a unique $($up to isomorphism$)$ indecomposable object $M^{*}\in \mathcal{A}$ with $M^{*}\not\cong M$ such that the additive subcategory of $\mathcal{A}$ generated by $\mathcal{T}_{M}$ and $M^{*}$, written as $\mu_{M}(\mathcal{T})$, belongs to $\mathfrak{C}$.

$(3)$~ The quiver of $\mu_{M}(\mathcal{T})$ is obtained from the quiver of $\mathcal{T}$ by the Fomin-Zelevinsky mutation at $M$ in the sense of {\rm [\ref{FZ}]}.

$(4)$~ There exist two exact triangles in $\mathcal{A}$ as follows:
$$\xymatrix{M \ar[r]^f & N \ar[r]^g & M^* \ar[r] & M[1]} ~~{\rm and}~~
\xymatrix{M^* \ar[r]^u & L \ar[r]^v& M \ar[r]& M^*[1],}$$
where $f,u$ are minimal left $\mathcal{T}_{M}$-approximations, and $g,v$ are minimal right $\mathcal{T}_{M}$-approximations in $\mathcal{A}$.
\end{definition}

The following notion is our main objective of study.

\begin{definition}\label{cluster category}
A $2$-Calabi-Yau triangulated $k$-category $\mathcal{A}$ is called a \emph{cluster category} if its cluster-tilting subcategories form a cluster structure.
\end{definition}

For the rest of this paper let $Q$ be the following quiver 
$$\xymatrix{ &2\ar[dl]
\ar[d]\ar[dr] && 4\ar[dl]\ar[dr] && 6\ar[dl]\ar[dr] &&
8\ar[dl]\ar@{.}[dr] &\\ 
0 & 1 & 3 && 5 && 7 && }
$$
of type $D_{\infty}$ with zigzag orientation, whose vertex set is written as $Q_{0}$ and ${\rm rep}(Q)$ be the category of finite dimensional k-linear representations of $Q$, whose Auslander-Reiten translation is written as $\tau_{Q}$. For each $x\in Q_{0}$, let $P_{x}$, $I_{x}$, and $S_{x}$ be the indecomposable projective representation, the indecomposable injective representation, and the simple representation at $x$, respectively. Recall from [\ref{RV}, (III.3)] that the Auslander-Reiten quiver $\Gamma_{{\rm rep}(Q)}$ of ${\rm rep}(Q)$ consists of the following three components, where 
$$ A_{n,m}: V_{i}=k ~~{\rm for}~~ 2\leq n\leq i \leq m.$$
$$ A_{m}^{(1)}: V_{i}=k ~~{\rm for}~~ i\leq m, i\neq 1, V_1=0.$$ 
$$ A_{m}^{(0)}: V_{i}=k ~~{\rm for}~~ i\leq m, i\neq 0, V_0=0.$$
$$B_{n,m}: V_{i}=k^{2} ~~{\rm for}~~ 2\leq i\leq n, V_{i}=k ~~{\rm for}~~ i=0,1 ~~{\rm and}~~ n+1\leq i\leq m.$$
(we always mean that $V_{j}=0$ at the vertices which are not mentioned).

\begin{figure}
$$ \xymatrix{
\ar@{.}[d]&&\ar@{.}[d]&&\ar@{.}[d]&\\       
  A_{5,5}\ar[dr] \ar@{--}[rr] && A_{3,7}\ar[dr] \ar@{--}[rr]&
& B_{1,9}\ar@{--}[r]&\\ 
& A_{3,5} \ar[dr]\ar[ur]\ar@{--}[rr]
&& B_{1,7}\ar[dr]\ar[ur]  \ar@{--}[r]&&\\ 
 A_{3,3}\ar[ur]\ar[dr] \ar@{--}[rr] &&
B_{1,5}\ar[ur]\ar[dr] \ar@{--}[rr]
&&B_{3,7}\ar@{--}[r]&\\ 
& B_{1,3} \ar[dr]\ar[ddr]\ar[ur]\ar@{--}[rr]
&& B_{3,5}\ar[dr]\ar[ddr]\ar[ur]  \ar@{--}[r]&&\\ 
 A^{(0)}_1\ar[ur] \ar@{--}[rr] &&
A^{(1)}_3\ar[ur] \ar@{--}[rr]
&& A^{(0)}_5\ar@{--}[r]&\\ 
 A_1^{(1)}\ar[uur] \ar@{--}[rr] &&
A_3^{(0)}\ar[uur] \ar@{--}[rr]
&& A^{(1)}_5\ar@{--}[r]&\\ 
} $$ \caption{The preprojective component $\mathcal{P}$ of $\Gamma_{{\rm rep}(Q)}$}
\end{figure}
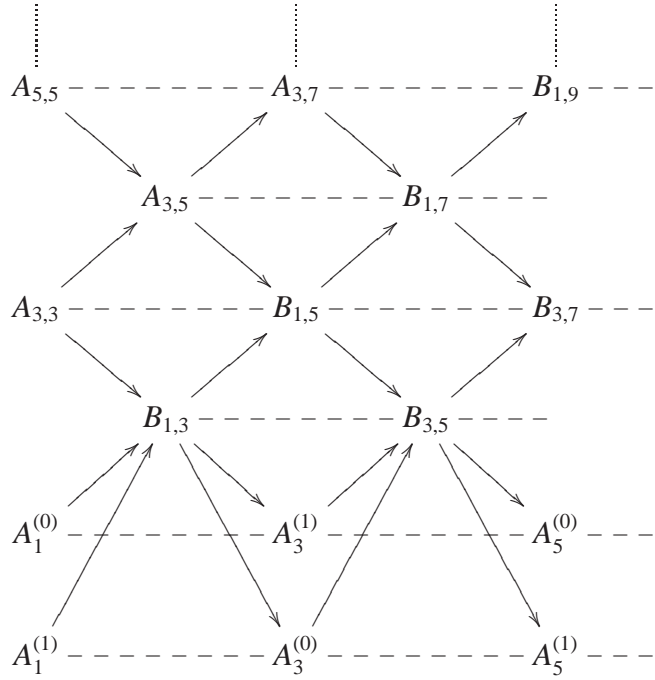
\begin{figure}
$$ \xymatrix{
&\ar@{.}[d]&&\ar@{.}[d]&&\ar@{.}[d]\\       
 \ar@{--}[r] & A_{2,10}\ar[dr] \ar@{--}[rr] && A_{4,8}\ar[dr]
 \ar@{--}[rr] && A_{6,6}\\ 
&\ar@{--}[r] & A_{2,8}\ar[ru]\ar[rd]\ar@{--}[rr]  && A_{4,6} \ar[ru]\ar[rd]%
\\ 
 \ar@{--}[r] & B_{2,8}\ar[dr]\ar[ru] \ar@{--}[rr] && A_{2,6}\ar[dr]\ar[ru]
 \ar@{--}[rr] && A_{4,4}\\ 
&\ar@{--}[r] & B_{2,6}\ar[ru]\ar[rd]\ar@{--}[rr]  && A_{2,4} \ar[ru]\ar[rd]%
\\ 
\ar@{--}[r] & B_{4,6}\ar[dr]\ar[ddr]\ar[ru] \ar@{--}[rr] &&
B_{2,4}\ar[dr]\ar[ddr]\ar[ru]
 \ar@{--}[rr] && A_{2,2}\\ 
&\ar@{--}[r] & A^{(1)}_4\ar[ru]\ar@{--}[rr]  && A_2^{(0)} \ar[ru]%
\\ 
&\ar@{--}[r] & A^{(0)}_4\ar[ruu]\ar@{--}[rr]  && A_2^{(1)} \ar[ruu]%
} $$ \caption{The preinjective component $\mathcal{I}$ of $\Gamma_{{\rm rep}(Q)}$}
\end{figure}
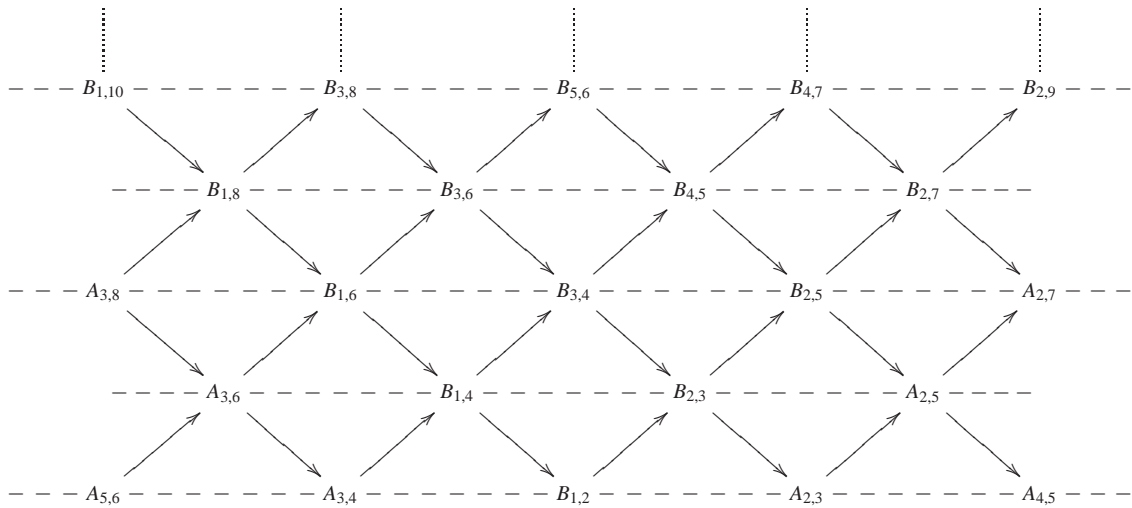
\begin{figure}
\tiny $$ \scriptscriptstyle \strut\hskip -1.5cm\xymatrix{
&\ar@{.}[d]&&\ar@{.}[d]&&\ar@{.}[d]&&\ar@{.}[d]&&\ar@{.}[d]&\\
\ar@{--}[r]& B_{1,10}\ar[dr] \ar@{--}[rr] && \ar[dr]B_{3,8}
\ar@{--}[rr] && \ar[dr] B_{5,6}\ar@{--}[rr] &&
B_{4,7}\ar[dr]\ar@{--}[rr]&& B_{2,9}\ar@{--}[r]&\\
& \ar@{--}[r] &B_{1,8} \ar[ur]\ar[dr] \ar@{--}[rr]&&
\ar[dr]\ar[ur]B_{3,6}\ar@{--}[rr] && \ar[dr]\ar[ur]
B_{4,5}\ar@{--}[rr]&& B_{2,7} \ar@{--}[r]\ar[ur]\ar[dr]&&\\
\ar@{--}[r] & A_{3,8}
 \ar@{--}[rr] \ar[ur]\ar[dr] && \ar[ur]\ar[dr] B_{1,6}\ar@{--}[rr] &&
\ar[ur]\ar[dr]B_{3,4} \ar@{--}[rr]
&& B_{2,5}\ar@{--}[rr]\ar[ur]\ar[dr]&&A_{2,7} \ar@{--}[r]&\\ 
&\ar@{--}[r]& A_{3,6} \ar@{--}[rr]\ar[dr]\ar[ur]&&
\ar[dr]\ar[ur]B_{1,4}\ar@{--}[rr] && \ar[dr]\ar[ur] B_{2,3}
\ar@{--}[rr]&&A_{2,5}\ar[ur]\ar[dr]\ar@{--}[r]&\\
\ar@{--}[r]& \ar@{--}[rr] A_{5,6}\ar[ur]&& A_{3,4}\ar[ur]
\ar@{--}[rr] && \ar[ur]B_{1,2} \ar@{--}[rr]
&& A_{2,3}\ar@{--}[rr]\ar[ur]&& A_{4,5}\ar@{--}[r]& 
} $$ \caption{The regular component $\mathcal{R}$ of $\Gamma_{{\rm rep}(Q)}$}
\end{figure}

\begin{remark}\label{GammaQ}
$(1)$~ All the Auslander-Reiten sequences in the preprojective component $\mathcal{P}$ can be summarized as follows.

$$0\rightarrow A_{i}^{(k)}\rightarrow B_{i,i+2}\rightarrow A_{i+2}^{(1-k)}\rightarrow 0, ~~i\geq 1, ~i~{\rm odd}, ~k\in \{0,1\};$$
$$0\rightarrow B_{i,i+2}\rightarrow A_{i+2}^{(0)}\oplus A_{i+2}^{(1)}\oplus B_{i,i+4}\rightarrow B_{i+2,i+4}\rightarrow 0, ~~i\geq 1, ~i~{\rm odd};$$
$$0\rightarrow A_{3,i}\rightarrow B_{1,i}\oplus A_{3,i+2}\rightarrow B_{1,i+2}\rightarrow 0, ~~i\geq 3, ~i~{\rm odd};$$
$$0\rightarrow A_{i,j}\rightarrow A_{i-2,j}\oplus A_{i,j+2}\rightarrow A_{i-2,j+2}\rightarrow 0, ~~5\leq i\leq j, ~i,j~{\rm odd};$$
$$0\rightarrow B_{i,j}\rightarrow B_{i+2,j}\oplus B_{i,j+2}\rightarrow B_{i+2,j+2}\rightarrow 0, ~~1\leq i<j-2, ~i,j~{\rm odd}.$$

$(2)$~ All the Auslander-Reiten sequences in the preinjective component $\mathcal{I}$ can be summarized as follows.

$$0\rightarrow B_{2,4}\rightarrow A_{2}^{(0)}\oplus A_{2}^{(1)}\oplus A_{2,4}\rightarrow A_{2,2}\rightarrow 0;$$
$$0\rightarrow A_{i+2}^{(k)}\rightarrow B_{i,i+2}\rightarrow A_{i}^{(1-k)}\rightarrow 0, ~~i\geq 2, ~i~{\rm even}, ~k\in \{0,1\};$$
$$0\rightarrow B_{i+2,i+4}\rightarrow A_{i+2}^{(0)}\oplus A_{i+2}^{(1)}\oplus B_{i,i+4}\rightarrow B_{i,i+2}\rightarrow 0, ~~i\geq 4, ~i~{\rm even};$$
$$0\rightarrow B_{2,i+2}\rightarrow B_{2,i}\oplus A_{2,i+2}\rightarrow A_{2,i}\rightarrow 0, ~~i\geq 4, ~i~{\rm even};$$
$$0\rightarrow A_{i-2,j+2}\rightarrow A_{i-2,j}\oplus A_{i,j+2}\rightarrow A_{i,j}\rightarrow 0, ~~4\leq i\leq j, ~i,j~{\rm even};$$
$$0\rightarrow B_{i+2,j+2}\rightarrow B_{i+2,j}\oplus B_{i,j+2}\rightarrow B_{i,j}\rightarrow 0, ~~2\leq i<j-2, ~i,j~{\rm even}.$$

$(3)$~ All the Auslander-Reiten sequences in the regular component $\mathcal{R}$ can be summarized as follows.

$$0\rightarrow A_{3,4}\rightarrow B_{1,4}\rightarrow B_{1,2}\rightarrow 0;$$
$$0\rightarrow B_{1,2}\rightarrow B_{2,3}\rightarrow A_{2,3}\rightarrow 0;$$
$$0\rightarrow B_{i,i+3}\rightarrow B_{i+2,i+3}\oplus B_{i,i+1}\rightarrow B_{i+1,i+2}\rightarrow 0, ~~i\geq 1, ~i~{\rm odd};$$
$$0\rightarrow B_{i+1,i+2}\rightarrow B_{i+2,i+3}\oplus B_{i,i+1}\rightarrow B_{i,i+3}\rightarrow 0, ~~i\geq 2, ~i~{\rm even};$$
$$0\rightarrow A_{i+2,i+3}\rightarrow A_{i,i+3}\rightarrow A_{i,i+1}\rightarrow 0, ~~i\geq 3, ~i~{\rm odd};$$
$$0\rightarrow A_{i,i+1}\rightarrow A_{i,i+3}\rightarrow A_{i+2,i+3}\rightarrow 0, ~~i\geq 2, ~i~{\rm even};$$
$$0\rightarrow A_{i+2,j+2}\rightarrow A_{i,j+2}\oplus A_{i+2,j}\rightarrow A_{i,j}\rightarrow 0, ~~3\leq i\leq j-3, ~i~{\rm odd},~j~{\rm even};$$
$$0\rightarrow A_{i,j}\rightarrow A_{i,j+2}\oplus A_{i+2,j}\rightarrow A_{i+2,j+2}\rightarrow 0, ~~2\leq i\leq j-3, ~i~{\rm even}, ~j~{\rm odd};$$
$$0\rightarrow A_{3,j+2}\rightarrow B_{1,j+2}\oplus A_{3,j}\rightarrow B_{1,j}\rightarrow 0, ~~j\geq 4, ~j~{\rm even};$$
$$0\rightarrow B_{2,j}\rightarrow B_{2,j+2}\oplus A_{2,j}\rightarrow A_{2,j+2}\rightarrow 0, ~~j\geq 3, ~j~{\rm odd};$$
$$0\rightarrow B_{i-2,j+2}\rightarrow B_{i,j+2}\oplus B_{i-2,j}\rightarrow B_{i,j}\rightarrow 0, ~~3\leq i<j, ~i~{\rm odd}, ~j~{\rm even};$$
$$0\rightarrow B_{i,j}\rightarrow B_{i,j+2}\oplus B_{i-2,j}\rightarrow B_{i-2,j+2}\rightarrow 0, ~~4\leq i<j, ~i~{\rm even}, ~j~{\rm odd}.$$

$(4)$~ It is not hard to see that the sequences in $(1),(2),(3)$ above are Auslander-Reiten sequences by considering a large enough subcategory of representations of a quiver of type $D_{n}$.

We also see that the preprojective component $\mathcal{P}$ contains the $A_{n,m}, A_{m}^{(0)}, A_{m}^{(1)}$ and $B_{n,m}$ with $n,m$ odd. The preinjective component $\mathcal{I}$ contains the $A_{n,m}, A_{m}^{(0)}, A_{m}^{(1)}$ and $B_{n,m}$ with $n,m$ even. The regular component $\mathcal{R}$ contains the $A_{n,m}$ and $B_{n,m}$ with $n+m$ odd.
\end{remark}

Let $\mathcal{C}$ be a Hom-finite Krull-Schmidt additive $k$-category and $\Gamma_{\mathcal{C}}$ be its Auslander-Reiten quiver. Then a connected component $\Gamma$ of $\Gamma_{\mathcal{C}}$ is called \emph{standard} if the mesh category $k(\Gamma)$ is equivalent to the full subcategory $\mathcal{C}(\Gamma)$ of $\mathcal{C}$ generated by the objects lying in $\Gamma$; see [\ref{LP1}]. It is well known that all the preprojective component $\mathcal{P}$, the preinjective component $\mathcal{I}$ and the regular component $\mathcal{R}$ are standard with ${\rm Hom}_{{\rm rep}(Q)}(\mathcal{I}, \mathcal{P})=0, {\rm Hom}_{{\rm rep}(Q)}(\mathcal{I}, \mathcal{R})=0$ and ${\rm Hom}_{{\rm rep}(Q)}(\mathcal{R}, \mathcal{P})=0$; see, for example, [\ref{BLP}, \ref{LP2}],

Furthermore, It is also known that the bounded derived category $D^{b}({\rm rep}(Q))$ of ${\rm rep}(Q)$ is a Hom-finite Krull-Schmidt triangulated $k$-category having almost split triangles, whose Auslander-Reiten translation is written as $\tau_{D}$; see [\ref{LP2}]. Setting $F=\tau_{D}^{-1}\circ [1]$, one obtains the canonical orbit category $\mathcal{C}(Q)=D^{b}({\rm rep}(Q))/F.$

Recall that the objects of $\mathcal{C}(Q)$ are those of $D^{b}({\rm rep}(Q))$; and for any objects $X,Y$, the morphisms are given by 
$${\rm Hom}_{\mathcal{C}(Q)}(X, Y)=\bigoplus\limits_{i\in \mathbb{Z}} {\rm Hom}_{D^{b}({\rm rep}(Q))}(X, F^{i}Y).$$
The composition of morphisms is given by 
$$(g_{i})_{i\in \mathbb{Z}}\circ (f_{i})_{i\in \mathbb{Z}}=(h_{i})_{i\in \mathbb{Z}},$$
where $h_{i}=\sum\limits_{p+q=i} F^{p}(g_{q})\circ f_{p}$. 

\begin{remark}\label{CQ}
$(1)$~ The canonical orbit category $\mathcal{C}(Q)$ is a Hom-finite Krull-Schmidt 2-Calabi-Yau triangulated $k$-category and there exists a canonical projection functor $\pi: D^{b}({\rm rep}(Q))\rightarrow \mathcal{C}(Q)$ which acts identically on the objects, and sends a morphism $f: X\rightarrow Y$ in ${\rm Hom}_{D^{b}({\rm rep}(Q))}(X, Y)$ to 
$$(f_{i})_{i\in \mathbb{Z}}\in \bigoplus\limits_{i\in \mathbb{Z}}{\rm Hom}_{D^{b}({\rm rep}(Q))}(X, F^{i}Y)={\rm Hom}_{\mathcal{C}(Q)}(X, Y),$$
where $f_{i}=f$ if $i=0$; and otherwise $f_{i}=0$. In particular, $\pi$ is a triangle functor and faithful; see {\rm [\ref{K}, \ref{LP2}]}.

$(2)$~ The Auslander-Reiten quiver of $\mathcal{C}(Q)$ is written as $\Gamma_{\mathcal{C}(Q)}$ and the Auslander-Reiten translation is written as $\tau_{\mathcal{C}}$. The objects of $D^{b}({\rm rep}(Q))$ lying in the connecting component $\mathcal{P}\sqcup \mathcal{I}[-1]$ of $\Gamma_{D^{b}({\rm rep}(Q))}$ or a regular component $\mathcal{R}$ of $\Gamma_{{\rm rep}(Q)}$ form a \emph{fundamental domain} of $\mathcal{C}(Q)$, denoted by $\mathcal{F}(Q)$, that is, every indecomposable object in $\mathcal{C}(Q)$ is isomorphic to a unique object in $\mathcal{F}(Q)$.
\end{remark}

For convenience, we state the following well known result.

\begin{lemma}\label{uf}
Let $X,Y$ be two representations lying in $\Gamma_{{\rm rep}(Q)}$.

$(1)$~ ${\rm Hom}_{{\rm rep}(Q)}(X,Y)\cong {\rm Hom}_{{\rm rep}(Q)}(\tau_{Q}^{-1}X, \tau_{Q}^{-1}Y).$

$(2)$~ ${\rm Hom}_{\mathcal{C}(Q)}(X,Y)\cong {\rm Hom}_{D^{b}({\rm rep}(Q))}(X,Y)\oplus D {\rm Hom}_{D^{b}({\rm rep}(Q))}(Y, \tau_{D}^{2}X).$

$(3)$~ Assume that $X\in \mathcal{P}$ and $Y\in \mathcal{R}\cup \mathcal{I}$. Then ${\rm Hom}_{\mathcal{C}(Q)}(X,Y)\cong {\rm Hom}_{{\rm rep}(Q)}(X, Y).$
\end{lemma}

\begin{proof}
We need only to prove Statement $(3)$, since all other parts are known; see [\ref{LP2}]. Assume that $X\in \mathcal{P}$ and $Y\in \mathcal{R}\cup \mathcal{I}$. Then by (2) we have 
$$\begin{array}{rcl}
{\rm Hom}_{\mathcal{C}(Q)}(X,Y) & \cong & {\rm Hom}_{D^{b}({\rm rep}(Q))}(X,Y)\oplus D {\rm Hom}_{D^{b}({\rm rep}(Q))}(Y, \tau_{D}^{2}X) \\
                                                        & \cong & {\rm Hom}_{{\rm rep}(Q))}(X,Y)\oplus D {\rm Hom}_{D^{b}({\rm rep}(Q))}(Y, \tau_{D}^{2}X).
\end{array}$$

If $\tau_{D}^{2}X$ is also a representation, then $D {\rm Hom}_{D^{b}({\rm rep}(Q))}(Y, \tau_{D}^{2}X)\in D{\rm Hom}_{{\rm rep}(Q)}(\mathcal{R}\cup \mathcal{I}, \mathcal{P})=0$. Otherwise, there exists a representation $Z\in \Gamma_{{\rm rep}(Q)}$ such that $\tau_{D}^{2}X=Z[-i]$, $i>0$. Thus in this case we also have $D {\rm Hom}_{D^{b}({\rm rep}(Q))}(Y, \tau_{D}^{2}X)=D {\rm Hom}_{D^{b}({\rm rep}(Q))}(Y, Z[-i])=0$. The proof of the lemma is completed.
\end{proof}

\section{Morphisms in the category ${\rm rep}(Q)$}

This section provides detailed information on three kinds of morphisms of the category ${\rm rep}(Q)$, that is, the morphisms between two regular representations, the morphisms between two preprojective representations and the morphisms from preprojective representation to regular representation. 

Let us start with the morphisms between two representations lying in the regular component $\mathcal{R}$ of $\Gamma_{{\rm rep}(Q)}$. 

Let $X$ be a representation lying in $\mathcal{R}$. Recall from [\ref{LP2}] that $X$ is called \emph{quasi-simple} if it has only one immediate predecessor in $\mathcal{R}$. Moreover, the \emph{forward rectangle $\mathcal{R}^{X}$} of $X$ is defined to be the full subquiver of $\mathcal{R}$ generated by its successors $Y$ such that, for any path $p: X\rightsquigarrow Y$ and any factorization $p=vu$ with paths $u: X\rightsquigarrow Z$ and $v: Z\rightsquigarrow Y$, either $u$ is sectional, or else, $Z$ has two distinct immediate predecessors. The \emph{backward rectangle $\mathcal{R}_{X}$} of $X$ is defined in a dual manner. Then we have the following result.

\begin{lemma}\label{RtoR}
Let $X, Y$ be two regular representations lying in $\Gamma_{{\rm rep}(Q)}$. Then ${\rm Hom}_{{\rm rep}(Q)}(X,Y)\neq 0$ if and only if $Y\in \mathcal{R}^X$ if and only if $X\in \mathcal{R}_Y$. In this case, moreover, ${\rm dim}_{k}{\rm Hom}_{{\rm rep}(Q)}(X,Y)=1$.
\end{lemma}

\begin{proof}
Since the regular component $\mathcal{R}$ of $\Gamma_{{\rm rep}(Q)}$ is standard of shape $\mathbb{Z}\mathbb{A}_{\infty}$, it follows from  [\ref{LP2}, Proposition 1.3] at once. The proof of the lemma is completed.
\end{proof}

Next we consider the morphisms between two representations lying in the preprojective component $\mathcal{P}$ of $\Gamma_{{\rm rep}(Q)}$. 

Recall also from [\ref{Y}] that a representation $X$ lying in $\Gamma_{{\rm rep}(Q)}$ is called \emph{boundary representation} if it has at most one direct predecessor and at most one direct successor in $\Gamma_{{\rm rep}(Q)}$, which is a generalization of the \emph{quasi-simple representation} defined only in the regular component $\mathcal{R}$.

Let $X$ be a representation lying in $\mathcal{P}$. We define the \emph{pseudo forward rectangle $\mathcal{PR}^{X}$} of $X$ to be the full subquiver of $\Gamma_{{\rm rep}(Q)}$ generated by its successors $Y$ such that, for any path $p: X\rightsquigarrow Y$, $p$ does not pass through any boundary representation. The \emph{pseudo backward rectangle $\mathcal{PR}_{X}$} of $X$ is defined in a dual manner. Then we have the following result.

\begin{lemma}\label{PtoP}
Let $X, Y$ be two preprojective representations lying in $\Gamma_{{\rm rep}(Q)}$. Let $X$ lies in the $\tau_{\mathcal{Q}}$-orbit of $P_{t}, t\geq 0$. 

$(1)$~  If $t=0,1$, say $X=A_{l}^{(k)}, l\geq 1, l~{\rm odd}, k\in \{0, 1\}$, then  \[
{\rm dim}_{k}{\rm Hom}_{{\rm rep}(Q)}(X, Y)=\left\{
\begin{array}{ll}
0 &\mbox{$Y\in \{Z~|~Z$ is not a successor of $X\}\cup\{A_{l'}^{(1-k)}~|~ l'>l, l'$~{\rm odd}$\};$}\\
1 &\mbox{$Y\in \{Z~|~Z$ is a successor of $X\}\backslash \{A_{l'}^{(1-k)}~|~ l'>l, l'$~{\rm odd}$\}.$}
\end{array}
\right.
\] 

$(2)$~ If $t\geq 2$, then  \[
{\rm dim}_{k}{\rm Hom}_{{\rm rep}(Q)}(X, Y)=\left\{
\begin{array}{ll}
0 &\mbox{$Y\in \{Z~|~Z$ is not a successor of $X\};$}\\
1 &\mbox{$Y\in \mathcal{PR}_{X} \cup(\{Z~|~Z$ is a successor of $X\}\cap \{A_{l}^{(k)}~|~ l\geq 1, l~{\rm odd}, k\in\{0, 1\}\});$}\\
2 &\mbox{$Y\in \{Z~|~Z$ is a successor of $X\}\backslash (\mathcal{PR}_{X} \cup \{A_{l}^{(k)}~|~ l\geq 1, l~{\rm odd}, k\in \{0, 1\}\}).$}
\end{array}
\right.
\] 
\end{lemma}

\begin{proof}
$(1)$~ Let $X, Y$ be preprojective representations lying in $\Gamma_{{\rm rep}(Q)}$. Let $X$ lies in the $\tau_{\mathcal{Q}}$-orbit of $P_{t}, t=0, 1$. If $Y$ is not a successor of $X$ in $\mathcal{P}$, then ${\rm Hom}_{{\rm rep}(Q)}(X, Y)=0$ since $\mathcal{P}$ is standard. Otherwise since $\tau_{Q}~A_{l+2}^{(k)}=A_{l}^{(1-k)}, l\geq 1, l~{\rm odd}$, by Lemma \ref{uf} we shall consider only the case where $X=P_{0}=A_{1}^{(1)}$. Since $Y$ is a successor of $P_{0}$, it follows that $Y=P_{0}, Y=B_{i, j}, i<j, i,j~{\rm odd}$ or $Y=A_{l}^{(k)}, l>1, l~{\rm odd}, k\in \{0, 1\}$. Hence in this case we have
$${\rm dim}_{k}{\rm Hom}_{{\rm rep}(Q)}(P_{0},Y)=0 \Leftrightarrow Y\in \{A_{l}^{(0)}~|~l>1,l~{\rm odd}\};$$
$$\begin{array}{rcl}
{\rm dim}_{k}{\rm Hom}_{{\rm rep}(Q)}(P_{0},Y)=1 & \Leftrightarrow & Y\in \{P_{0}\}\cup\{B_{i, j}~|~ i<j, i,j~{\rm odd}\}\cup\{A_{l}^{(1)}~|~l>1,l~{\rm odd}\};\\
                                                                                & \Leftrightarrow & Y\in \{Z~|~Z$ is a successor of $P_{0}\}\backslash 
                                                                                \{A_{l}^{(0)}~|~l>1,l~{\rm odd}\}.
\end{array}$$

Statement $(1)$ is established.

$(2)$~ Similarly we may assume that $X=P_{t}, t\geq 2$. Then all the successors of $P_{t}$ in $\mathcal{P}$ have the following forms:
$$\mathcal{PR}_{P_{t}}=\{B_{i,j}~|~1\leq i<t\leq j, i,j~{\rm odd}\}\cup\{A_{k,l}~|~1<k\leq t\leq l, k,l~{\rm odd}\};$$
$$S_{1}=\{A_{l}^{(k)}~|~t\leq l, l~{\rm odd}, k\in \{0,1\}\}~~~~{\rm and}~~~~S_{2}=\{B_{i,j}~|~t\leq i<j, i,j~{\rm odd}\}.$$
Hence in this case we have 
$$\begin{array}{rcl}
{\rm dim}_{k}{\rm Hom}_{{\rm rep}(Q)}(P_{t},Y)=1 & \Leftrightarrow & Y\in\{B_{i,j}~|~1\leq i<t\leq j, i,j~{\rm odd}\}~\cup \\
                                                                               &                         &  \{A_{k,l}~|~1<k\leq t\leq l, k,l~{\rm odd}\}\cup
                                                                                \{A_{l}^{(k)}~|~t\leq l, l~{\rm odd}, k\in \{0, 1\}\} \\
                                                                               & \Leftrightarrow & Y\in \mathcal{PR}_{P_{t}}\cup S_{1} \\
                                                                               & \Leftrightarrow & Y\in \mathcal{PR}_{P_{t}} \cup(\{Z~|~Z$ is a successor of $P_{t}\}\cap \{A_{l}^{(k)}~|~ l\geq 1, l~{\rm odd}, k\in\{0, 1\}\});                
\end{array}$$
$$\begin{array}{rcl}
{\rm dim}_{k}{\rm Hom}_{{\rm rep}(Q)}(P_{t},Y)=2 & \Leftrightarrow & Y\in \{B_{i,j}~|~t\leq i<j, i,j~{\rm odd}\} \\
                                                                               & \Leftrightarrow & Y\in S_{2} \\
                                                                               & \Leftrightarrow & Y\in \{Z~|~Z$ is a successor of $P_{t}\}\backslash (\mathcal{PR}_{P_{t}}\cup S_{1}) \\
                                                                               & \Leftrightarrow & Y\in \{Z~|~Z$ is a successor of $P_{t}\}\backslash (\mathcal{PR}_{P_{t}} \cup \{A_{l}^{(k)}~|~ l\geq 1, l~{\rm odd}, k\in \{0, 1\}\}).
\end{array}$$

Since $\{Z~|~Z$ is a successor of $P_{t}$ in $\mathcal{P}\}=\mathcal{PR}_{P_{t}}\cup S_{1}\cup S_{2}$ and $\mathcal{P}$ is standard, it follows that 
$${\rm dim}_{k}{\rm Hom}_{{\rm rep}(Q)}(P_{t},Y)=0 \Leftrightarrow Z {\rm ~is ~not ~a ~successor ~of~} P_{t}$$
at once. The proof of the lemma is completed.
\end{proof}

For the rest of this section, we shall concentrate on the morphisms from $X$ to $Y$, where $X$ lies in the preprojective component $\mathcal{P}$ and $Y$ lies in the regular component $\mathcal{R}$ of $\Gamma_{{\rm rep}(Q)}$. 

Let $S$ be a quasi-simple representation lying in $\mathcal{R}$. Since $\mathcal{R}$ is of shape $\mathbb{Z}\mathbb{A}_{\infty}$, $\mathcal{R}$ has a unique ray starting in $S$, written as $(S\rightarrow)$, and a unique co-ray ending in $S$ written as $(\rightarrow S)$. We denote by $\mathcal{W}(S)$ the full subquiver of $\mathcal{R}$ generated by the representations $Z$ for which there exist paths $M\rightsquigarrow Z\rightsquigarrow N$, where $M\in (\rightarrow S)$ and $N\in (S\rightarrow)$, and call it the \emph{infinite wing} with \emph{wing vertex} $S$; see [\ref{LP2}].

The following is analogous to the $A_{\infty}^{\infty}$ case; compare [\ref{LP2}, Proposition 2.6].

\begin{proposition}\label{CtoR}
Let $X$ be a preprojective representation lying in $\Gamma_{{\rm rep}(Q)}$. Let $X$ lies in the $\tau_{\mathcal{Q}}$-orbit of $P_{t}, t\geq 0$. 

$(1)$~ If $t=0,1$, then the regular component $\mathcal{R}$ of $\Gamma_{{\rm rep}(Q)}$ has a unique quasi-simple $Z$ such that, for any $Y\in \mathcal{R}$, we have 
\[
{\rm dim}_{k}{\rm Hom}_{{\rm rep}(Q)}(X,Y)=\left\{
\begin{array}{ll}
1 &\mbox{$Y\in \mathcal{W}(Z);$}\\
0 &\mbox{$Y\notin \mathcal{W}(Z).$}
\end{array}
\right.
\] 

$(2)$~ If $t\geq 2$, then the regular component $\mathcal{R}$ of $\Gamma_{{\rm rep}(Q)}$ has two quasi-simples $Z_{1}, Z_{2}$ such that, for any $Y\in \mathcal{R}$, we have 
\[
{\rm dim}_{k}{\rm Hom}_{{\rm rep}(Q)}(X,Y)=\left\{
\begin{array}{ll}
2 &\mbox{$Y\in \mathcal{W}(Z_{1})\cap \mathcal{W}(Z_{2});$}\\
1 &\mbox{$Y\in (\mathcal{W}(Z_{1})\cup \mathcal{W}(Z_{2}))\backslash (\mathcal{W}(Z_{1})\cap \mathcal{W}(Z_{2}));$}\\
0 &\mbox{$Y\notin \mathcal{W}(Z_{1})\cup \mathcal{W}(Z_{2}).$}
\end{array}
\right.
\] 
\end{proposition}

\begin{proof}
$(1)$~ Let $X$ be a preprojective representation lying in $\Gamma_{{\rm rep}(Q)}$. Let $X$ lies in the $\tau_{\mathcal{Q}}$-orbit of $P_{t}, t=0$ or $1$. Since $\mathcal{R}$ is $\tau_{Q}$-stable, by Lemma \ref{uf} we may assume that $X=P_{t}$, $t=0$ or $1$. We shall consider only the case where $t=0$. Then there exists a unique quasi-simple representation $Z=B_{1,2}\in \mathcal{R}$ such that ${\rm Hom}_{{\rm rep}(Q)}(P_{0},Z)\neq 0$. Observe that $\mathcal{W}(B_{1,2})=\{B_{i,j}~|~ 1\leq i<j, i+j ~{\rm odd}\}$ and $\mathcal{R}\backslash \mathcal{W}(B_{1,2})=\{A_{k,l}~|~ 2\leq k<l, k+l~{\rm odd}\}$. Thus for any $Y\in \mathcal{R}$, we have 
$$\begin{array}{rcl}
{\rm dim}_{k}{\rm Hom}_{{\rm rep}(Q)}(P_{0},Y)=1 & \Leftrightarrow & Y\in \{B_{i,j}~|~1\leq i<j, i+j ~{\rm odd}\} \\
                                                                                & \Leftrightarrow & Y\in \mathcal{W}(B_{1,2}),
\end{array}$$
$$\begin{array}{rcl}
{\rm dim}_{k}{\rm Hom}_{{\rm rep}(Q)}(P_{0},Y)=0 & \Leftrightarrow & Y\in \{A_{k,l}~|~ 2\leq k<l, k+l ~{\rm odd}\} \\ 
                                                                                & \Leftrightarrow & Y\notin \mathcal{W}(B_{1,2}).
\end{array}$$

$(2)$~ Similarly, we assume that $X=P_{t}, t\geq 2$. Suppose first that $t=2$. Then there exist two quasi-simples $Z_{1}=B_{1,2}$ and $Z_{2}=A_{2,3}$ such that ${\rm Hom}_{{\rm rep}(Q)}(P_{2},Z_{1})\neq 0$ and ${\rm Hom}_{{\rm rep}(Q)}(P_{2},Z_{2})\neq 0$. It is easy to see that $\mathcal{W}(A_{2,3})=\{B_{i,j}~|~2\leq i<j, i+j ~{\rm odd}\}\cup\{A_{2,l}~|~3\leq l, l~{\rm odd}\}$. Thus for any $Y\in \mathcal{R}$, we have 
$$\begin{array}{rcl}
{\rm dim}_{k}{\rm Hom}_{{\rm rep}(Q)}(P_{2},Y)=2 & \Leftrightarrow & Y\in \{B_{i,j}~|~ 2\leq i<j, i+j ~{\rm odd}\} \\
                                                                                & \Leftrightarrow & Y\in \mathcal{W}(B_{1,2})\cap \mathcal{W}(A_{2,3}),
\end{array}$$
$$\begin{array}{rcl}
{\rm dim}_{k}{\rm Hom}_{{\rm rep}(Q)}(P_{2},Y)=1 & \Leftrightarrow & Y\in \{B_{1,j}~|~ 2\leq j, j~{\rm even}\}\cup\{A_{2,l}~|~ 3\leq l, l~{\rm odd}\} \\
                                                                                & \Leftrightarrow & Y\in (\mathcal{W}(B_{1,2})\cup \mathcal{W}(A_{2,3}))\backslash (\mathcal{W}(B_{1,2})\cap \mathcal{W}(A_{2,3})),                                                                            
\end{array}$$ 
$$\begin{array}{rcl}
{\rm dim}_{k}{\rm Hom}_{{\rm rep}(Q)}(P_{2},Y)=0 & \Leftrightarrow & Y\in \{A_{k,l}~|~ 3\leq k<l, k+l ~{\rm odd}\} \\
                                                                                & \Leftrightarrow & Y \notin \mathcal{W}(B_{1,2})\cup \mathcal{W}(A_{2,3}).
\end{array}$$

Suppose now that $t>2$. We shall consider only the case where $t$ is odd. Then there exist two quasi-simples $Z_{1}=A_{t,t+1}$ and $Z_{2}=A_{t-1,t}$ such that ${\rm Hom}_{{\rm rep}(Q)}(P_{t},Z_{1})\neq 0$ and ${\rm Hom}_{{\rm rep}(Q)}(P_{t},Z_{2})\neq 0$. Observe that 
$$\begin{array}{rcl}
\mathcal{W}(A_{t,t+1}) & = & \{B_{i_{1},j_{1}}~|~1\leq i_{1}< t\leq j_{1}, i_{1} ~{\rm odd}, j_{1} ~{\rm even}\} ~\cup\\
                                     &    & \{B_{i_{2},j_{2}}~|~t\leq i_{2}< j_{2}, i_{2}+j_{2}~{\rm odd}\} ~\cup\\
                                     &    & \{A_{k,l}~|~3\leq k\leq t\leq l, k~{\rm odd}, l~{\rm even}\},
\end{array}$$
$$\begin{array}{rcl}
\mathcal{W}(A_{t-1,t}) & = & \{B_{i_{1},j_{1}}~|~2\leq i_{1}< t\leq j_{1}, i_{1} ~{\rm even}, j_{1}~{\rm odd}\} 
~\cup\\
                                    &    & \{B_{i_{2},j_{2}}~|~t\leq i_{2}<j_{2}, i_{2}+j_{2} ~{\rm odd}\} ~\cup\\
                                    &    & \{A_{k,l}~|~2\leq k\leq t\leq l, k~{\rm even}, l~{\rm odd}\}.
\end{array}$$
Thus for any $Y\in \mathcal{R}$, we have 
$$\begin{array}{rcl}
{\rm dim}_{k}{\rm Hom}_{{\rm rep}(Q)}(P_{t},Y)=2 & \Leftrightarrow & Y\in \{B_{i,j}~|~ t\leq i<j, i+j~{\rm odd}\} \\
                                                                               & \Leftrightarrow & Y\in \mathcal{W}(A_{t,t+1})\cap \mathcal{W}(A_{t-1,t});
\end{array}$$
$$\begin{array}{rcl}
{\rm dim}_{k}{\rm Hom}_{{\rm rep}(Q)}(P_{t},Y)=1 & \Leftrightarrow & Y\in \{B_{i,j}~|~ 1\leq i<t\leq j, i+j ~{\rm odd}\}~\cup \\ 
                                                                               &                         & \{A_{k,l}~|~2\leq k\leq t\leq l, k+l ~{\rm odd}\}\\
                                                                               & \Leftrightarrow & Y\in (\mathcal{W}(A_{t,t+1})\cup \mathcal{W}(A_{t-1,t}))\backslash (\mathcal{W}(A_{t,t+1})\cap \mathcal{W}(A_{t-1,t}));
\end{array}$$
$$\begin{array}{rcl}
{\rm dim}_{k}{\rm Hom}_{{\rm rep}(Q)}(P_{t},Y)=0 & \Leftrightarrow & Y\in \{B_{i,j}~|~1\leq i<j<t, i+j~{\rm odd}\}~\cup \\
                                                                               &                         & \{A_{k_{1},l_{1}}~|~2\leq k_{1}<l_{1}<t, k_{1}+l_{1}~{\rm odd}\}~\cup \\
                                                                               &                         & \{A_{k_{2},l_{2}}~|~t<k_{2}<l_{2}, k_{2}+l_{2}~{\rm odd}\} \\
                                                                               & \Leftrightarrow & Y\notin \mathcal{W}(A_{t,t+1})\cup \mathcal{W}(A_{t-1,t}).
\end{array}$$
The proof of the proposition is completed.
\end{proof}

We end this section with the following example, which describes the distribution of $k$-dimension of the {\rm Hom}-space ${\rm Hom}_{{\rm rep}(Q)}(X, Y)$ in detail when $X$ is preprojective and $Y\in \mathcal{P}\cup \mathcal{R}$.

\begin{example}\label{kdim}
Let $X=P_{5}=A_{5,5}$. Then the distribution of ${\rm dim}_{k}{\rm Hom}_{{\rm rep}(Q)}(X, -)$ is shown as follows.

\begin{figure}[h] \centering

  \includegraphics*[150,10][555,525]{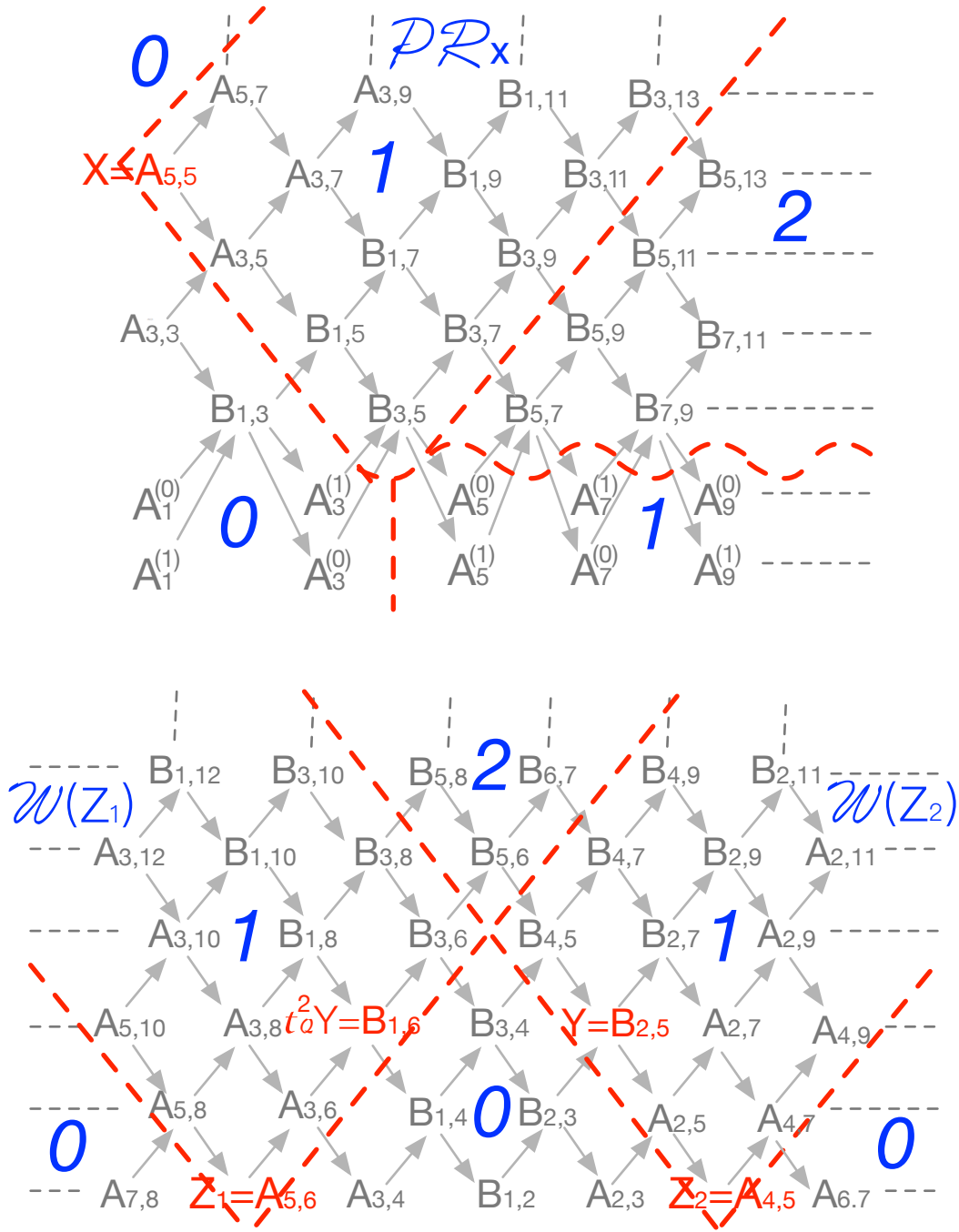}

  \caption{The distribution of ${\rm dim}_{k}{\rm Hom}_{{\rm rep}(Q)}(P_{5}, -)$}
\end{figure}
\end{example}

\section{The cluster category $\mathcal{C}(Q)$}

This section proves our main result, that is, the orbit category $\mathcal{C}(Q)$ is a cluster category. Recall from Remark \ref{CQ} that every indecomposable object in $\mathcal{C}(Q)$ is isomorphic to a unique object in the fundamental domain $\mathcal{F}(Q)=\mathcal{P}\sqcup \mathcal{I}[-1]\sqcup \mathcal{R}$. We begin with the following lemma.

\begin{lemma}\label{Rok}
Let $X, Y$ be two objects lying in the regular component $\mathcal{R}$. If ${\rm Ext}_{\mathcal{C}(Q)}^{1}(X,Y)=0$, then either ${\rm Hom}_{\mathcal{C}(Q)}(X,Y)=0$ or ${\rm Hom}_{\mathcal{C}(Q)}(Y,X)=0$.
\end{lemma}

\begin{proof}
Suppose to the contrary that ${\rm Hom}_{\mathcal{C}(Q)}(X,Y)\neq 0$ and ${\rm Hom}_{\mathcal{C}(Q)}(Y,X)\neq 0$. Since both $X$ and $Y$ are regular representations, by Lemma \ref{uf} we have 
$${\rm Hom}_{\mathcal{C}(Q)}(X,Y)\cong {\rm Hom}_{{\rm rep}(Q)}(X,Y)\oplus D {\rm Hom}_{{\rm rep}(Q)}(Y, \tau_{Q}^{2}X)\neq 0,$$
$${\rm Hom}_{\mathcal{C}(Q)}(Y,X)\cong {\rm Hom}_{{\rm rep}(Q)}(Y,X)\oplus D {\rm Hom}_{{\rm rep}(Q)}(X, \tau_{Q}^{2}Y)\neq 0,$$
$$\begin{array}{rcl}
{\rm Ext}_{\mathcal{C}(Q)}^{1}(X,Y) & \cong & {\rm Hom}_{\mathcal{C}(Q)}(X,\tau_{\mathcal{C}}Y) \\
                                                   & \cong & {\rm Hom}_{{\rm rep}(Q)}(X, \tau_{Q}Y)\oplus D {\rm Hom}_{{\rm rep}(Q)}(Y, \tau_{Q}X)\\
                                                   & = & 0.
\end{array}$$

Assume first that ${\rm Hom}_{{\rm rep}(Q)}(X,Y)\neq 0$. Since the regular component $\mathcal{R}$ of $\Gamma_{{\rm rep}(Q)}$ is acyclic and standard, there exists a path $X\rightsquigarrow Y$ in $\mathcal{R}$ and hence there exists no path $Y\rightsquigarrow X$ in $\mathcal{R}$. Thus ${\rm Hom}_{{\rm rep}(Q)}(Y,X)=0$ and hence ${\rm Hom}_{{\rm rep}(Q)}(X, \tau_{Q}^{2}Y)\neq 0$. By Lemma \ref{RtoR} it follows that both $Y$ and $\tau_{Q}^{2}Y$ are in the forward rectangle of $X$ and thus so is $\tau_{Q}Y$. Now by Lemma \ref{RtoR} again we have ${\rm Hom}_{{\rm rep}(Q)}(X, \tau_{Q}Y)\neq 0$, which is a contradiction.

Assume now that ${\rm Hom}_{{\rm rep}(Q)}(Y, \tau_{Q}^{2}X)\neq 0$. If ${\rm Hom}_{{\rm rep}(Q)}(Y,X)\neq 0$, then similarly we have $\tau_{Q}X$ is in the forward rectangle of $Y$ and thus ${\rm Hom}_{{\rm rep}(Q)}(Y, \tau_{Q}X)\neq 0$, which is a contradiction. Otherwise, we have ${\rm Hom}_{{\rm rep}(Q)}(X, \tau_{Q}^{2}Y)\neq 0$ and hence there exists a path $X\rightsquigarrow \tau_{Q}^{2}Y\rightsquigarrow Y\rightsquigarrow \tau_{Q}^{2}X\rightsquigarrow X$ in $\mathcal{R}$, a contradiction. The proof of the lemma is completed.
\end{proof}

In Lemma \ref{Rok} we see that if both $X$ and $Y$ are indecomposable objects lying in the regular component $\mathcal{R}$ such that ${\rm Ext}_{\mathcal{C}(Q)}^{1}(X,Y)=0$, then either ${\rm Hom}_{\mathcal{C}(Q)}(X, Y)=0$ or ${\rm Hom}_{\mathcal{C}(Q)}(Y, X)=0$. However, it is not true in general. 

\begin{lemma}\label{C det R}
Let $X$ be an object lying in the connecting component $\mathcal{P}\sqcup \mathcal{I}[-1]$ and in the $\tau_{\mathcal{C}}$-orbit of $P_{t}, t\geq 0$. Let $Y$ be an object lying in the regular component $\mathcal{R}$. Assume that ${\rm Ext}_{\mathcal{C}(Q)}^{1}(X, Y)=0$. If ${\rm Hom}_{\mathcal{C}(Q)}(X, Y)\neq 0$ and ${\rm Hom}_{\mathcal{C}(Q)}(Y, X)\neq 0$, then $t\geq 3$ and $Y$ is uniquely determined by $X$. In this case, moreover, ${\rm dim}_{k}{\rm Hom}_{\mathcal{C}(Q)}(X, Y)=1$ and ${\rm dim}_{k}{\rm Hom}_{\mathcal{C}(Q)}(Y, X)=1$.
\end{lemma}

\begin{proof}
Since the Auslander-Reiten translation $\tau_{\mathcal{C}}$ for $\mathcal{C}(Q)$ coincides with the shift functor, we may assume that $X=P_{t}, t\geq 0$. Since $Y$ lies in the regular component $\mathcal{R}$, by Lemma \ref{uf} we have
$${\rm Hom}_{\mathcal{C}(Q)}(P_{t},Y) \cong {\rm Hom}_{{\rm rep}(Q)}(P_{t},Y)\neq 0,$$
$$\begin{array}{rcl}
{\rm Hom}_{\mathcal{C}(Q)}(Y, P_{t}) & \cong & {\rm Hom}_{{\rm rep}(Q)}(Y, P_{t})\oplus D {\rm Hom}_{{\rm rep}(Q)}(P_{t}, \tau_{Q}^{2}Y) \\
                                                           & = & D {\rm Hom}_{{\rm rep}(Q)}(P_{t}, \tau_{Q}^{2}Y)\neq 0,
\end{array}$$
$${\rm Ext}_{\mathcal{C}(Q)}^{1}(P_{t},Y) \cong {\rm Hom}_{\mathcal{C}(Q)}(P_{t},\tau_{\mathcal{C}}Y) \cong {\rm Hom}_{{\rm rep}(Q)}(P_{t}, \tau_{Q}Y)=0.$$

If $t=0$ or $1$, then by Proposition \ref{CtoR} it follows that $Y, \tau_{Q}^{2}Y\in \mathcal{W}(B_{1,2})$ while $\tau_{Q}Y\notin \mathcal{W}(B_{1,2})$, a contradiction. 

Suppose now that $t\geq 2$. By Proposition \ref{CtoR} again that the regular component $\mathcal{R}$ of $\Gamma_{{\rm rep}(Q)}$ has two quasi-simples $Z_{1}, Z_{2}$ such that $Y,\tau_{Q}^{2}Y\in \mathcal{W}(Z_{1})\cup \mathcal{W}(Z_{2})$ while $\tau_{Q} Y\notin \mathcal{W}(Z_{1})\cup \mathcal{W}(Z_{2})$. Then it is easy to see that in the $t=2$ case such $Y$ can not be exist and in the $t\geq 3$ case, we have \[
Y=\left\{
\begin{array}{lll}
A_{2,2t-3} &\mbox{$t=3,4;$}\\
B_{t-3,t} &\mbox{$t\geq 5, t~{\rm odd};$}\\
B_{t-4,t+1} &\mbox{$t\geq 6, t~{\rm even}.$}
\end{array}
\right.
\] 
This shows that $Y$ is unique.

We may further assume that $Z_{1}=\tau_{Q}^{s}~Z_{2}, s>0$. Then we have $\tau_{Q}^{2}Y\in \mathcal{W}(Z_{1})\backslash \mathcal{W}(Z_{2})$ and $Y\in \mathcal{W}(Z_{2})\backslash \mathcal{W}(Z_{1})$. Finally by Proposition \ref{CtoR} it follows that ${\rm dim}_{k}{\rm Hom}_{\mathcal{C}(Q)}(P_{t}, Y)=1$ and ${\rm dim}_{k}{\rm Hom}_{\mathcal{C}(Q)}(Y, P_{t})=1$. The proof of the lemma is completed.
\end{proof}

Let $\mathcal{T}$ be a weakly cluster-tilting subcategory of $\mathcal{C}(Q)$. Then we need more notions to determine whether an object is in $\mathcal{T}$ or not; compare [\ref{HJ}].

\begin{definition}\label{forbidden region}
Let $X$ be an object lying in the fundamental domain $\mathcal{F}(Q)$. Then we define the \emph{forbidden region} of $X$ as $H(X)=\{Y\in \mathcal{F}(Q)~|~{\rm Ext}_{\mathcal{C}(Q)}^{1}(Y, X)\neq 0\}$.
\end{definition}

Now we have the following criterion.

\begin{lemma}\label{for reg}
Let $\mathcal{T}$ be a weakly cluster-tilting subcategory of $\mathcal{C}(Q)$. Let $X_{1},X_{2},\cdots, X_{m}$ and $Y_{1},Y_{2},\cdots,Y_{n}$ be objects lying in the fundamental domain $\mathcal{F}(Q)$. If $X_{i}\in \mathcal{T}, i=1,2,\cdots,m$ and $Y_{j}\notin \mathcal{T}, j=1,2,\cdots,n$, then $H(Y_{j})\nsubseteq (\bigcup\limits_{i=1}^{m} H(X_{i}))\cup\{Y_{1}, Y_{2},\cdots, Y_{n}\}, 1\leq j\leq n$.
\end{lemma}

\begin{proof}
Suppose to the contrary that there exists $j$ such that $H(Y_{j})\subseteq (\bigcup\limits_{i=1}^{m} H(X_{i}))\cup\{Y_{1}, Y_{2},\cdots, Y_{n}\}$. Since $Y_{j}\notin \mathcal{T}$ and $\mathcal{T}$ is a weakly cluster-tilting category, by definition there exists an indecomposable object $Z\in \mathcal{T}$ such that ${\rm Ext}_{\mathcal{C}(Q)}^{1}(Z, Y_{j})\neq 0$. Thus $Z\in H(Y_{j})\subseteq (\bigcup\limits_{i=1}^{m} H(X_{i}))\cup\{Y_{1}, Y_{2},\cdots, Y_{n}\}$. Since $Y_{k}\notin \mathcal{T}, k=1,2,\cdots,n$, there exists $i$ such that $Z\in H(X_{i})$, i.e., ${\rm Ext}_{\mathcal{C}(Q)}^{1}(Z, X_{i})\neq 0$, which contradicts that both $Z$ and $X_{i}$ are in $\mathcal{T}$. The proof of the lemma is completed.
\end{proof}

The following notion will allow us to describe the forbidden region by the Auslander-Reiten quiver $\Gamma_{\mathcal{C}(Q)}$.

\begin{definition}\label{fb for region}
Let $X$ be an object lying in the connecting component $\mathcal{P}\sqcup \mathcal{I}[-1]$. Then we define the \emph{forward forbidden region} of $X$ as $H^{+}(X)=\{Y~|~Y$ is a successor, however, not a sectional successor of $X$ in $\mathcal{P}\sqcup \mathcal{I}[-1]\}$. That is, any path $p: X\rightsquigarrow Y$ is not sectional. Dually, we define the \emph{backward forbidden region} of $X$ as $H^{-}(X)=\{Y~|~Y$ is a predecessor, however, not a sectional predecessor of $X$ in $\mathcal{P}\sqcup \mathcal{I}[-1]\}$.
\end{definition}

\begin{proposition}\label{coincide}
Let $X$ be an object lying in the connecting component $\mathcal{P}\sqcup \mathcal{I}[-1]$. Let $X$ lies in the $\tau_{\mathcal{C}}$-orbit of $P_{t}, t\geq 0$.

$(1)$~ If $t=0,1$, say $X=A_{l}^{(k)}, l\geq 1, l~{\rm odd}, k\in \{0, 1\}$, then 
$$H(X)\cap (\mathcal{P}\sqcup \mathcal{I}[-1])=(H^{+}(X)\cup H^{-}(X))\backslash (\{A_{l+4}^{(k)}, A_{l+6}^{(k)},\cdots\}\cup\{\cdots, A_{4}^{(1-k)}[-1], A_{2}^{(1-k)}[-1], A_{1}^{(k)}, A_{3}^{(k)},\cdots, \tau_{\mathcal{C}}^{2}A_{l}^{(k)}\}).$$
In particular, $H(X)\cap (\mathcal{P}\sqcup \mathcal{I}[-1])\subseteq H^{+}(X)\cup H^{-}(X)$.

$(2)$~ If $t\geq 2$, then $H(X)\cap (\mathcal{P}\sqcup \mathcal{I}[-1])=H^{+}(X)\cup H^{-}(X).$
\end{proposition}

\begin{proof}
$(1)$~ We shall consider only the case where $X=A_{1}^{(1)}=P_{0}$. Let $Y\in H(P_{0})$. Suppose first that $Y\in \mathcal{I}[-1]$. Then by Lemma \ref{uf} we have
$${\rm Ext}_{\mathcal{C}(Q)}^{1}(Y, P_{0}) \cong D {\rm Ext}_{\mathcal{C}(Q)}^{1}(P_{0}, Y) \cong  D {\rm Hom}_{\mathcal{C}(Q)}(P_{0}, Y[1])  \cong D {\rm Hom}_{{\rm rep}(Q)}(P_{0}, Y[1]).$$                                     

Hence ${\rm Ext}_{\mathcal{C}(Q)}^{1}(Y, P_{0})\neq 0$ implies that $Y[1]\in \{B_{i,j}~|~2\leq i<j, i,j~{\rm even}\}\cup\{A_{l'}^{(1)}~|~l'\geq 2, l'~{\rm even}\}$.

Suppose next that $Y\in \mathcal{P}$. We will consider ${\rm Ext}_{\mathcal{C}(Q)}^{1}(Y, P_{0})\cong D {\rm Hom}_{\mathcal{C}(Q)}(P_{0}, \tau_{\mathcal{C}}Y)$ again. If $\tau_{\mathcal{C}}Y=I_{j}[-1], j\in Q_{0}$, then by definition we have 
$$\begin{array}{rcl}
{\rm Hom}_{\mathcal{C}(Q)}(P_{0}, \tau_{\mathcal{C}}Y) & \cong & \bigoplus\limits_{k\in \mathbb{Z}}{\rm Hom}_{D^{b}({\rm rep}(Q))}(P_{0}, (\tau^{-1}[1])^{k}~I_{j}[-1]) \\
                                                                                          & \cong & \cdots \oplus {\rm Hom}_{D^{b}({\rm rep}(Q))}(P_{0}, I_{j}[-1])\oplus {\rm Hom}_{D^{b}({\rm rep}(Q))}(P_{0}, P_{j}[1])\oplus \cdots \\
                                                                                          & = & 0.
\end{array}$$

If both $Y$ and $\tau_{\mathcal{C}}Y$ are in $\mathcal{P}$, then by Lemma \ref{uf} again we have
$$\begin{array}{rcl}
{\rm Hom}_{\mathcal{C}(Q)}(P_{0}, \tau_{\mathcal{C}}Y) & \cong & {\rm Hom}_{D^{b}({\rm rep}(Q))}(P_{0},\tau_{D}Y)\oplus D {\rm Hom}_{D^{b}({\rm rep}(Q))}(Y, \tau_{D}P_{0}) \\
                                                                                          & \cong & {\rm Hom}_{{\rm rep}(Q)}(P_{0}, \tau_{Q}Y).
\end{array}$$

Hence ${\rm Ext}_{\mathcal{C}(Q)}^{1}(Y, P_{0})\neq 0$ implies that $\tau_{Q}Y\in \{B_{i,j}~|~1\leq i<j, i,j~{\rm odd}\}\cup\{A_{l'}^{(1)}~|~l'\geq 1, l'~{\rm odd}\}$.

In summary, we obtain that 
$$\begin{array}{rcl}
H(P_{0})\cap (\mathcal{P}\sqcup \mathcal{I}[-1]) & = & \{B_{i,j}[-1]~|~2\leq i<j, i,j~{\rm even}\}\cup \{A_{l'}^{(1)}[-1]~|~l'\geq 2, l'~{\rm even}\}~\cup \\
                                                                             &    & \{B_{i,j}~|~3\leq i<j, i,j~{\rm odd}\}\cup \{A_{l''}^{(0)}~|~l''\geq 3, l''~{\rm odd}\}.
\end{array}$$

On the other hand, it is easy to see that 
$$H^{+}(P_{0})=\{B_{i,j}~|~3\leq i<j, i,j~{\rm odd}\}\cup \{A_{l'}^{(k')}~|~l'\geq 5, l'~{\rm odd}, k'\in\{0,1\}\}\cup \{A_{3}^{(0)}\};$$
$$H^{-}(P_{0})=\{B_{i,j}[-1]~|~2\leq i<j, i,j~{\rm even}\}\cup \{A_{l''}^{(k'')}[-1]~|~l''\geq 4, l''~{\rm even}, k''\in\{0,1\}\}\cup\{A_{2}^{(1)}[-1]\}.$$

This establishes our claim.

$(2)$~ We shall consider only the case where $X=P_{t}$ and $t$ is odd. Let $Y\in H(P_{t})$. Suppose first that $Y\in \mathcal{I}[-1]$. Similarly we have ${\rm Ext}_{\mathcal{C}(Q)}^{1}(Y, P_{t})\cong D{\rm Hom}_{{\rm rep}(Q)}(P_{t}, Y[1])$. Hence ${\rm Ext}_{\mathcal{C}(Q)}^{1}(Y, P_{t})\neq 0$ implies that 
$$Y[1]\in \{B_{i,j}~|~i< t\leq j, i,j~{\rm even}\}\cup \{A_{k,l}~|~k\leq t\leq l, k,l~{\rm even}\}\cup \{A_{l'}^{(k')}~|~t\leq l', l'~{\rm even}, k'\in \{0,1\}\}.$$

Suppose next that $Y\in \mathcal{P}$. If $\tau_{\mathcal{C}}Y=I_{j}[-1], j\in Q_{0}$, then similarly we can show that ${\rm Ext}_{\mathcal{C}(Q)}^{1}(Y, P_{t})=0$. If both $Y$ and $\tau_{\mathcal{C}}Y$ are in $\mathcal{P}$, then ${\rm Ext}_{\mathcal{C}(Q)}^{1}(Y, P_{t})\cong D{\rm Hom}_{{\rm rep}(Q)}(P_{t}, \tau_{Q}Y)\neq 0$ implies that 
$$\tau_{Q}Y\in \{B_{i,j}~|~t\leq j, i,j~{\rm odd}\}\cup \{A_{k,l}~|~k\leq t\leq l, k,l~{\rm odd}\}\cup \{A_{l'}^{(k')}~|~t\leq l', l'~{\rm odd}, k'\in \{0,1\}\}.$$

In summary, we obtain that 
$$\begin{array}{rcl}
H(P_{t})\cap (\mathcal{P}\sqcup \mathcal{I}[-1]) & = & \{B_{i,j}[-1]~|~i< t\leq j, i,j~{\rm even}\}\cup \{A_{k,l}[-1]~|~k\leq t\leq l, k,l~{\rm even}\}~\cup \\
                                                                            &    & \{A_{l'}^{(k')}[-1]~|~t\leq l', l'~{\rm even}, k'\in \{0,1\}\}\cup \{B_{i,j}~|~3\leq i<j, j\geq t+2, i,j~{\rm odd}\}~\cup \\
                                                                            &     & \{B_{1,j}~|~j\geq t+2, j~{\rm odd}\}\cup \{A_{k,l}~|~k\leq t-2<t+2\leq l, k,l~{\rm odd}\} ~\cup \\
                                                                            &     & \{A_{l'}^{(k')}~|~t+2\leq l', l'~{\rm odd}, k'\in \{0,1\}\}.
\end{array}$$

Finally, it is easy to see that 
$$\begin{array}{rcl}
H^{+}(P_{t}) & = & \{B_{i,j}~|~i<j, j\geq t+2, i,j~{\rm odd}\}\cup \{A_{k,l}~|~k\leq t-2<t+2\leq l, k,l~{\rm odd}\}~\cup \\
                    &    & \{A_{l'}^{(k')}~|~t+2\leq l', l'~{\rm odd}, k'\in \{0,1\}\};
\end{array}$$
$$\begin{array}{rcl}
H^{-}(P_{t}) & = & \{B_{i,j}[-1]~|~i< t\leq j, i,j~{\rm even}\}\cup \{A_{k,l}[-1]~|~k\leq t\leq l, k,l~{\rm even}\}~\cup \\
                    &    & \{A_{l'}^{(k')}[-1]~|~t\leq l', l'~{\rm even}, k'\in \{0,1\}\}.
\end{array}$$

Therefore, $H(P_{t})\cap (\mathcal{P}\sqcup \mathcal{I}[-1])=H^{+}(P_{t})\cup H^{-}(P_{t})$. The proof of the proposition is completed.
\end{proof}

Let $X$ be an object lying in the fundamental domain $\mathcal{F}(Q)$. Then $X$ is called \emph{boundary object} if it has only one immediate predecessor in $\Gamma_{\mathcal{C}(Q)}$. Moreover, an object $U$ is called \emph{boundary predecessor} of $X$ if $U$ is an boundary object and there exists a unique sectional path $U\rightsquigarrow X$ in $\Gamma_{\mathcal{C}(Q)}$. Dually, an object $V$ is called \emph{boundary successor} of $X$ if $V$ is an boundary object and there exists a unique sectional path $X\rightsquigarrow V$ in $\Gamma_{\mathcal{C}(Q)}$. Let $X$ lies in the $\tau_{\mathcal{C}}$-orbit of $P_{t}$. If $t\geq 2$, then it is easy to see that $X$ has exactly two boundary predecessors, denoted by $U_{1}, U_{2}$ and exactly two boundary successors, denoted by $V_{1}, V_{2}$, respectively. The following observation is important for our investigation.

\begin{theorem}\label{in t}
Let $X$ be an object lying in the connecting component $\mathcal{P}\sqcup \mathcal{I}[-1]$ and in the $\tau_{\mathcal{C}}$-orbit of $P_{t}, t\geq 0$. Let $Y$ be an object lying in the regular component $\mathcal{R}$. Assume that $\mathcal{T}$ is a weakly cluster-tilting subcategory of $\mathcal{C}(Q)$ and $X, Y\in \mathcal{T}$. If ${\rm Hom}_{\mathcal{C}(Q)}(X, Y)\neq 0$ and ${\rm Hom}_{\mathcal{C}(Q)}(Y, X)\neq 0$, then at least one of the boundary predecessors or boundary successors of $X$ is in $\mathcal{T}$.
\end{theorem}

\begin{proof}
By Lemma \ref{C det R} it follows that $t\geq 3$ and hence $X$ has two boundary predecessors $U_{1}, U_{2}$ and two boundary successors $V_{1}, V_{2}$ in $\mathcal{P}\sqcup \mathcal{I}[-1]$. Suppose to the contrary that none of the object in $\{U_{1}, U_{2}, V_{1}, V_{2}\}$ is in $\mathcal{T}$. We claim that $H(V_{1})\subseteq H(X)\cup H(Y)\cup\{U_{1}, U_{2}, V_{1}, V_{2}\}$. Then by Lemma \ref{for reg} we obtain our desired contradiction.

We shall consider only the case where $X=P_{t}$, $t$ is odd and $V_{1}=A_{t}^{(0)}$. Since 
$${\rm Ext}_{\mathcal{C}(Q)}^{1}(Z, A_{t}^{(0)})\cong D{\rm Hom}_{{\rm rep}(Q)}(A_{t}^{(0)}, \tau_{Q}Z)~~{\rm and}~~{\rm Ext}_{\mathcal{C}(Q)}^{1}(Z, P_{t})\cong D{\rm Hom}_{{\rm rep}(Q)}(P_{t}, \tau_{Q}Z)$$
for any $Z\in \mathcal{R}$, by Proposition \ref{CtoR} and Proposition \ref{coincide} it follows that 
$$H(A_{t}^{(0)})\subseteq H^{+}(A_{t}^{(0)})\cup H^{-}(A_{t}^{(0)})\cup \mathcal{W}(A_{t+1, t+2});$$
$$H(P_{t})=H^{+}(P_{t})\cup H^{-}(P_{t})\cup \mathcal{W}(\tau_{Q}^{-1}A_{t,t+1})\cup \mathcal{W}(A_{t+1,t+2}).$$

It is easy to see that $H^{+}(A_{t}^{(0)})\subseteq H^{+}(P_{t})$ and
$$\begin{array}{rcl}
H^{-}(A_{t}^{(0)}) & = & \{A_{l}^{(k)}~|~1\leq l\leq t-4, l~{\rm odd}, k\in \{0,1\}\}\cup \{A_{t-2}^{(1)}\}\cup \{A_{k,l}~|~k\leq l<t, k,l~{\rm odd}\}~\cup\\
                            &    & \{B_{i,j}~|~i<j<t, i,j~{\rm odd}\}\cup \{A_{k',l'}[-1]~|~k'<t, k'\leq l', k',l'~{\rm even}\}~\cup \\
                            &    & \{B_{i',j'}[-1]~|~i'<j', i',j'~{\rm even}\}\cup \{A_{l'}^{(k')}[-1]~|~l'~{\rm even}, k'\in \{0,1\}\},
\end{array}$$
$$\begin{array}{rcl}
H^{-}(P_{t}) & = & \{A_{k,l}[-1]~|~k<t<l, k,l~{\rm even}\}\cup \{B_{i,j}[-1]~|~t<j, i,j~{\rm even}\}~ \cup \\
                   &    & \{A_{l'}^{(k')}[-1]~|~l'>t, l'~{\rm even}, k'\in \{0,1\}\}.
\end{array}$$

Thus if $Z\in H(A_{t}^{(0)})\backslash H(P_{t})$, then $Z\in S_{1}\cup S_{2}\cup \cdots \cup S_{7}$, where
$$S_{1}=\{A_{l}^{(k)}~|~1\leq l\leq t-4, l~{\rm odd}, k\in \{0,1\}\},~~S_{2}=\{A_{t-2}^{(1)}\},$$
$$S_{3}=\{A_{k,l}~|~k\leq l<t, k,l~{\rm odd}\},~~S_{4}=\{B_{i,j}~|~i<j<t, i,j~{\rm odd}\},$$
$$S_{5}=\{A_{k',l'}[-1]~|~k'\leq l'\leq t, k',l'~{\rm even}\},~~S_{6}=\{B_{i',j'}[-1]~|~i'<j'\leq t, i',j'~{\rm even}\},$$
$$S_{7}=\{A_{l'}^{(k')}[-1]~|~l'\leq t, l'~{\rm even}, k'\in \{0,1\}\}.$$

Suppose first that $t=3$. As seen in the proof of Lemma \ref{C det R}, $Y=A_{2,3}$. Let $Z\in H(A_{3}^{(0)})\backslash H(P_{3})$. Then $S_{1}=S_{3}=S_{4}=S_{6}=\emptyset, S_{2}=\{A_{1}^{(1)}\}, S_{5}=\{A_{2,2}[-1]\}$ and $S_{7}=\{A_{2}^{(0)}[-1],A_{2}^{(1)}[-1]\}$. Since 
$${\rm Ext}_{\mathcal{C}(Q)}^{1}(A_{1}^{(1)}, A_{2,3})\cong {\rm Hom}_{{\rm rep}(Q)}(A_{1}^{(1)}, B_{1,2})\neq 0,$$
$${\rm Ext}_{\mathcal{C}(Q)}^{1}(A_{2,2}[-1], A_{2,3})\cong D{\rm Hom}_{{\rm rep}(Q)}(A_{2,3}, A_{2,2})\neq 0$$
and $S_{7}=\{A_{2}^{(0)}[-1],A_{2}^{(1)}[-1]\}=\{U_{1},U_{2}\}$, it follows that $H(A_{3}^{(0)})\backslash H(P_{3})\subseteq H(A_{2,3})\cup \{A_{2}^{(0)}[-1],A_{2}^{(1)}[-1]\}$ and hence 
$$H(A_{3}^{(0)})\subseteq H(P_{3})\cup H(A_{2,3})\cup \{A_{2}^{(0)}[-1],A_{2}^{(1)}[-1], A_{3}^{(0)}, A_{3}^{(1)}\}.$$ 

Our claim is established.

Suppose now that $t\geq 5$. As seen in the proof of Lemma \ref{C det R}, $Y=B_{t-3,t}$ and it is easy to see that 
$$\tau_{Q}^{\mathbb{Z}}B_{t-3,t}=\{\cdots,A_{5,2t},A_{3,2t-2}, B_{1,2t-4}, B_{3,2t-6}, \cdots, B_{t-2,t-1}, B_{t-3,t}, B_{t-5,t+2},\cdots, B_{2,2t-5},A_{2,2t-3},A_{4,2t-1},\cdots\}.$$ 

Let $Z_{1}\in S_{1}\cup S_{2}\cup S_{3}\cup S_{4}$. Then it is easy to see that $Z_{1}$ can be written as $\tau_{\mathcal{C}}^{-s}P_{r}, r\in Q_{0}$, where $0\leq r\leq t-2$ and $0\leq s\leq \frac{t-r-2}{2}$. By Lemma \ref{uf} we have
$${\rm Ext}_{\mathcal{C}(Q)}^{1}(Z_{1}, B_{t-3, t})\cong {\rm Hom}_{{\rm rep}(Q)}(Z_{1}, \tau_{Q}B_{t-3, t})\cong {\rm Hom}_{{\rm rep}(Q)}(P_{r}, \tau_{Q}^{s}B_{t-2,t-1}).$$

Observe that $t-2-2s\geq r\geq 0$ and $t$ is odd. Thus $t-2-2s\geq 1$ and hence $\tau_{Q}^{s}B_{t-2,t-1}\cong B_{t-2-2s, t-1+2s}$. Since $t-1+2s>t-2-2s\geq r$, it follows that 
$${\rm Ext}_{\mathcal{C}(Q)}^{1}(Z_{1}, B_{t-3, t})\cong {\rm Hom}_{{\rm rep}(Q)}(P_{r}, B_{t-2-2s, t-1+2s})\neq 0$$

Therefore, $S_{1}\cup S_{2}\cup S_{3}\cup S_{4}\subseteq H(B_{t-3,t})$.

On the other hand, let $Z_{2}\in S_{5}\cup S_{6}\cup S_{7}$. Then $Z_{2}$ can be written as $\tau_{\mathcal{C}}^{s
}I_{r}[-1], r\in Q_{0}$, where $r=0, 0\leq s\leq \frac{t-3}{2}$ and $0<r\leq t-1$ and $0\leq s\leq \frac{t-r-1}{2}$. Now by Lemma \ref{uf} again we have 
$${\rm Ext}_{\mathcal{C}(Q)}^{1}(Z_{2}, B_{t-3, t})\cong D{\rm Hom}_{\mathcal{C}(Q)}(B_{t-3,t}, \tau_{\mathcal{C}}^{s
}I_{r})\cong D{\rm Hom}_{{\rm rep}(Q)}(\tau_{Q}^{-s}B_{t-3,t}, I_{r}).$$

Let $r=0$. Then $t-3-2s\geq 0$, $t$ is odd and hence $t-3-2s$ is even. If $t-3-2s=0$, then $Z_{2}=\tau_{\mathcal{C}}^{\frac{t-3}{2}}I_{0}[-1]\in \{\tau_{\mathcal{C}}^{\frac{t-3}{2}}I_{0}[-1], \tau_{\mathcal{C}}^{\frac{t-3}{2}}I_{1}[-1]\}$. If $t-3-2s\geq 2$, then $\tau_{Q}^{-s}B_{t-3,t}=B_{t-3-2s, t+2s}$ and ${\rm Ext}_{\mathcal{C}(Q)}^{1}(Z_{2}, B_{t-3, t})\cong D{\rm Hom}_{{\rm rep}(Q)}(B_{t-3-2s, t+2s}, I_{0})\neq 0$.

Next let $r>0$. Then $t-1-2s\geq r>0$, $t$ is odd and hence $t-1-2s\geq 2$ is even. If $t-1-2s\geq 4$, then $\tau_{Q}^{-s}B_{t-3,t}=B_{t-3-2s, t+2s}$ and $t+2s>t-1-2s\geq r$. It follows that ${\rm Ext}_{\mathcal{C}(Q)}^{1}(Z_{2}, B_{t-3, t})\cong D{\rm Hom}_{{\rm rep}(Q)}(B_{t-3-2s, t+2s}, I_{r})\neq 0$. If $t-1-2s=2$, then $\tau_{Q}^{-s}B_{t-3,t}=A_{2,2t-3}$ and hence ${\rm Ext}_{\mathcal{C}(Q)}^{1}(Z_{2}, B_{t-3, t})\cong D{\rm Hom}_{{\rm rep}(Q)}(A_{2,2t-3}, I_{r})$. In the first case where $r=1$, $Z_{2}=\tau_{\mathcal{C}}^{\frac{t-3}{2}}I_{1}[-1]\in \{\tau_{\mathcal{C}}^{\frac{t-3}{2}}I_{0}[-1], \tau_{\mathcal{C}}^{\frac{t-3}{2}}I_{1}[-1]\}$. In the second case where $r\geq 2$, we have $2\leq r\leq t-1-2s=2$ and hence $r=2< 2t-3$. It follows that ${\rm Ext}_{\mathcal{C}(Q)}^{1}(Z_{2}, B_{t-3, t})\cong D{\rm Hom}_{{\rm rep}(Q)}(A_{2,2t-3}, I_{r})\neq 0$. 

Therefore, $S_{5}\cup S_{6}\cup S_{7}\subseteq H(B_{t-3,t})\cup \{\tau_{\mathcal{C}}^{\frac{t-3}{2}}I_{0}[-1], \tau_{\mathcal{C}}^{\frac{t-3}{2}}I_{1}[-1]\}$.

Finally, for any $Z\in H(A_{t}^{(0)})\backslash H(P_{t})$, $Z\in S_{1}\cup S_{2}\cup \cdots \cup S_{7}\subseteq H(B_{t-3,t})\cup \{\tau_{\mathcal{C}}^{\frac{t-3}{2}}I_{0}[-1], \tau_{\mathcal{C}}^{\frac{t-3}{2}}I_{1}[-1]\}$. Thus,
$$H(A_{t}^{(0)})\subseteq H(P_{t})\cup H(B_{t-3,t})\cup \{\tau_{\mathcal{C}}^{\frac{t-3}{2}}I_{0}[-1], \tau_{\mathcal{C}}^{\frac{t-3}{2}}I_{1}[-1],A_{t}^{(0)},A_{t}^{(1)}\}.$$
The proof of the theorem is completed.
\end{proof}

We illustrate Theorem \ref{in t} with the following example. 

\begin{example}\label{S1-S7}
Let $X=P_{5}=A_{5,5}$. Then the distribution of the sets $S_{1}, S_{2},\cdots, S_{7}$ in the proof of {\rm Theorem \ref{in t}} is shown as follows.

\begin{figure}[h] \centering

  \includegraphics*[90,113][520,380]{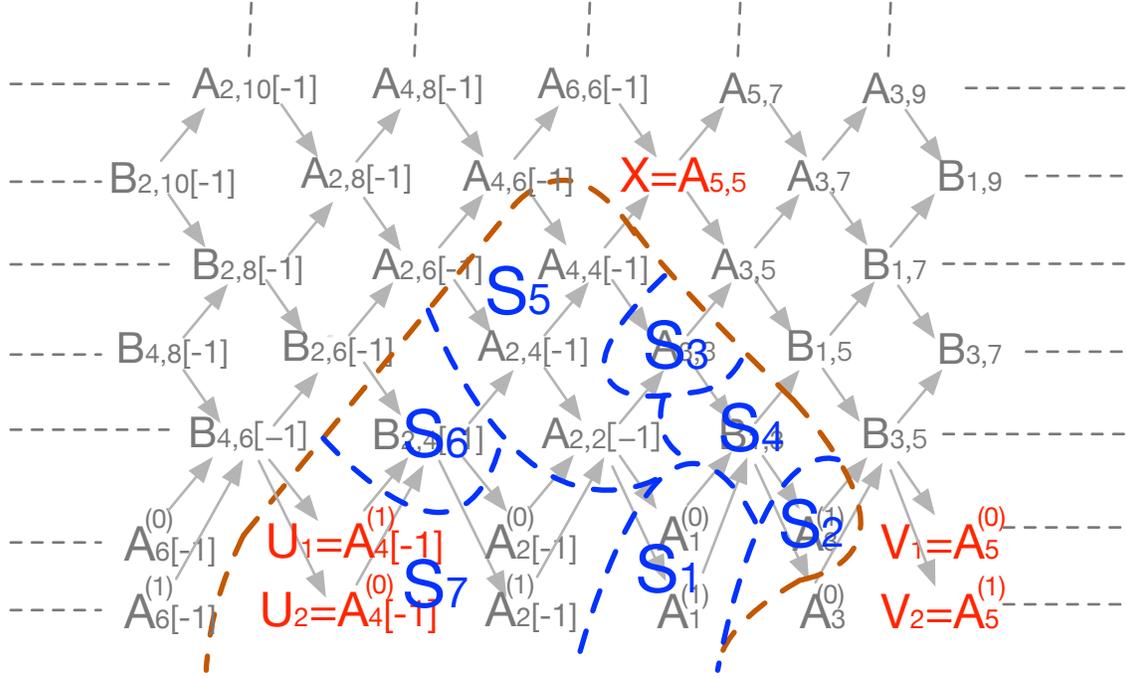}

  \caption{The distribution of $S_{1}, S_{2},\cdots, S_{7}$ in {\rm Theorem} \ref{in t} when $X=P_{5}$}
\end{figure}
\end{example}

Let $X, Y$ be two representations lying in $\Gamma_{{\rm rep}(Q)}$ and $f: X\rightarrow Y$ be a morphism in ${\rm Hom}_{{\rm rep}(Q)}(X, Y)$. For the rest of this paper, we shall denote by $f_{D}=\iota(f)$ and $f_{\mathcal{C}}=\pi\circ \iota(f)$, where $\iota: {\rm rep}(Q)\rightarrow D^{b}({\rm rep}(Q))$ is the natural embedding functor and $\pi: D^{b}({\rm rep}(Q))\rightarrow \mathcal{C}(Q)$ is the canonical projection functor.

Next we shall prove that in the above setting, any morphism in ${\rm Hom}_{\mathcal{C}(Q)}(Y, X)$ will factor through the boundary predecessor $U_{i}$ of $X$ and any morphism in ${\rm Hom}_{\mathcal{C}(Q)}(X, Y)$ will factor through the boundary successor $V_{i}$ of $X$, $i=1,2$. 

\begin{lemma}\label{f t}
Let $X$ be an object lying in the connecting component $\mathcal{P}\sqcup \mathcal{I}[-1]$ and in the $\tau_{\mathcal{C}}$-orbit of $P_{t}, t\geq 0$. Let $Y$ be an object lying in the regular component $\mathcal{R}$. Assume that ${\rm Ext}_{\mathcal{C}(Q)}^{1}(X,Y)=0$. If ${\rm Hom}_{\mathcal{C}(Q)}(X, Y)\neq 0$ and ${\rm Hom}_{\mathcal{C}(Q)}(Y, X)\neq 0$, then 

$(1)$~ $X$ has two boundary predecessors $U_{1}, U_{2}$ and two boundary successors $V_{1}, V_{2}$ in the connecting component $\mathcal{P}\sqcup \mathcal{I}[-1]$.

$(2)$~ Each morphism $X\rightarrow Y$ factors through $V_{i}$ in $\mathcal{C}(Q)$, $i=1,2$.

$(3)$~ Each morphism $Y\rightarrow X$ factors through $U_{i}$ in $\mathcal{C}(Q)$, $i=1,2$.
\end{lemma}

\begin{proof}
By Lemma \ref{C det R} we have $t\geq 3$ and hence Statement $(1)$ follows at once. For Statement $(2)$, we shall consider only the case where $X=P_{t}$, $t\geq 3$, $t$ is odd and $V_{1}=A_{t}^{(0)}$. By Lemma \ref{C det R}, ${\rm dim}_{k}{\rm Hom}_{\mathcal{C}(Q)}(P_{t}, Y)=1$. Thus in order to show that each morphism $P_{t}\rightarrow Y$ factors through $A_{t}^{(0)}$, it suffices to find two morphisms in ${\rm Hom}_{\mathcal{C}(Q)}(P_{t}, A_{t}^{(0)})$ and ${\rm Hom}_{\mathcal{C}(Q)}(A_{t}^{(0)},Y)$ such that the composition is non-zero in ${\rm Hom}_{\mathcal{C}(Q)}(P_{t}, Y)$.

Define $f=(f_{i})_{i\geq 0}\in {\rm Hom}_{{\rm rep}(Q)}(P_{t}, A_{t}^{(0)})$, where $f_{t}={\rm id}_{k}$ and $f_{s}=0, s\neq t$. Next we define $g=(g_{i})_{i\geq 0}\in {\rm Hom}_{{\rm rep}(Q)}(A_{t}^{(0)}, Y)$ as follows. If $t=3$, then by Lemma \ref{C det R} we have $Y=A_{2,3}$. Thus we choose $g_{2}=g_{3}={\rm id}_{k}$ and $g_{0}=g_{1}=g_{s}$, $s\geq 4$. If $t\geq 5$, then by Lemma \ref{C det R} again we have $Y=B_{t-3, t}$. Thus we choose $g_{1}=g_{t-2}=g_{t-1}=g_{t}={\rm id}_{k}, g_{2}=g_{3}=\cdots=g_{t-3}=\left(\begin{array}{c}
0\\
1
\end{array}\right)$ and $g_{0}=g_{s}=0$, $s\geq t+1$. Since in both cases $(g\circ f)_{t}={\rm id}_{k}\neq 0$, it follows that $g\circ f$ is a non-zero morphism in ${\rm Hom}_{{\rm rep}(Q)}(P_{t}, Y)$. Observe that both functors $\iota$ and $\pi$ are faithful. As a consequence, the induced morphism $g_{\mathcal{C}}\circ f_{\mathcal{C}}: \xymatrix{P_{t}\ar[r]^{f_{\mathcal{C}}} &A_{t}^{(0)} \ar[r]^{g_{\mathcal{C}}} & Y}$ is non-zero in ${\rm Hom}_{\mathcal{C}(Q)}(P_{t}, Y)$. Statement $(2)$ is established.

In the proof of Statement $(3)$, we shall consider the case where $X=\tau_{\mathcal{C}}^{2}P_{t}$, $t\geq 3$, $t$ is odd and $U_{1}=A_{t+3}^{(0)}[-1]$. By Lemma \ref{C det R}, ${\rm dim}_{k}{\rm Hom}_{\mathcal{C}(Q)}(Y, \tau_{\mathcal{C}}^{2}P_{t})=1$. Thus in order to show that each morphism $Y\rightarrow \tau_{\mathcal{C}}^{2}P_{t}$ factors through $A_{t+3}^{(0)}[-1]$, it suffices to find two morphism in ${\rm Hom}_{\mathcal{C}(Q)}(Y, A_{t+3}^{(0)}[-1])$ and ${\rm Hom}_{\mathcal{C}(Q)}(A_{t+3}^{(0)}[-1], \tau_{\mathcal{C}}^{2}P_{t})$ such that the composition is non-zero in ${\rm Hom}_{\mathcal{C}(Q)}(Y, \tau_{\mathcal{C}}^{2}P_{t})$.

Define $g=(g_{i})_{i\geq 0}\in {\rm Hom}_{{\rm rep}(Q)}(A_{t+1}^{(1)}, I_{t})$, where $g_{t-1}=g_{t}=g_{t+1}={\rm id}_{k}$ and $g_{s}=0, s\neq t-1,t,t+1$. Next we define $f=(f_{i})_{i\geq 0}\in {\rm Hom}_{{\rm rep}(Q)}(Y, A_{t+1}^{(1)})$ as follows. If $t=3$, then by Lemma \ref{C det R} we have $Y=A_{3, 4}$. Thus we choose $f_{3}=f_{4}={\rm id}_{k}$ and $f_{s}=0, s\neq 3,4$. If $t\geq 5$, then by Lemma \ref{C det R} again we have $Y=B_{t-4, t+1}$. Thus we choose $f_{0}=f_{t-3}=f_{t-2}=f_{t-1}=f_{t}=f_{t+1}={\rm id}_{k}$, $f_{2}=f_{3}=\cdots=f_{t-4}=\left(\begin{array}{cc}
1 & 0
\end{array}\right)$ and $f_{1}=f_{s}=0, s\geq t+2$. Since in both cases $(g\circ f)_{t+1}={\rm id}_{k}\neq 0$, it follows that $g\circ f$ is a non-zero morphism in ${\rm Hom}_{{\rm rep}(Q)}(Y, I_{t})$ and hence $g_{D}\circ f_{D}$ is a non-zero morphism in ${\rm Hom}_{D^{b}({\rm rep}(Q))}(Y, I_{t})$.

Let $F=\tau_{D}^{-1}\circ [1]$. Then it is easy to see that $F(\tau_{\mathcal{C}}^{2}P_{t})=I_{t}$ and $F(A_{t+3}^{(0)}[-1])=A_{t+1}^{(1)}$. Finally we define $u\in {\rm Hom}_{\mathcal{C}(Q)}(Y, A_{t+3}^{(0)}[-1])$ as $u_{1}=f_{D}$, $u_{s}=0, s\neq 1$ and $v\in {\rm Hom}_{\mathcal{C}(Q)}(A_{t+3}^{(0)}[-1], \tau_{\mathcal{C}}^{2}P_{t})$ as $v_{0}=F^{-1}(g_{D})$, $v_{s}=0, s\neq 0$. Therefore, by definition, 
$$(v\circ u)_{1}=\bigoplus\limits_{p+q=1}(F^{p}(v_{q})\circ u_{p})=F(v_{0})\circ u_{1}=g_{D}\circ f_{D}\neq 0$$
and hence $\xymatrix{Y\ar[r]^{u}& A_{t+3}^{(0)}[-1]\ar[r]^{v} &\tau_{\mathcal{C}}^{2}P_{t}}$ is non-zero in ${\rm Hom}_{\mathcal{C}(Q)}(Y, \tau_{\mathcal{C}}^{2}P_{t})$. The proof of the lemma is completed.
\end{proof}

Thus, the following statement is an immediate consequence of Theorem \ref{in t} and Lemma \ref{f t}.

\begin{corollary}\label{PRok}
Let $\mathcal{T}$ be a weakly cluster-tilting subcategory of $\mathcal{C}(Q)$. Let $X, Y\in \mathcal{T}$ be two objects with $X$ lying in the connecting component $\mathcal{P}\sqcup \mathcal{I}[-1]$ and $Y$ lying in the regular component $R$. Then for any morphisms $f: X\rightarrow Y$ and $g: Y\rightarrow X$ in $\mathcal{C}(Q)$ with $f,g\neq 0$, there exists $Z\in \mathcal{T}, Z\ncong X,Y$ such that either $f$ or $g$ factors through $Z$ in $\mathcal{C}(Q)$.
\end{corollary}

Furthermore, if both $X$ and $Y$ are lying in the connecting component $\mathcal{P}\sqcup \mathcal{I}[-1]$ such that ${\rm Ext}_{\mathcal{C}(Q)}^{1}(X, Y)=0$, then ${\rm Hom}_{\mathcal{C}(Q)}(X, Y)\neq 0$ and ${\rm Hom}_{\mathcal{C}(Q)}(Y, X)\neq 0$ implies that both $X$ and $Y$ are boundary objects.

\begin{lemma}\label{force bo}
Let $X, Y$ be two objects lying in the connecting component $\mathcal{P}\sqcup \mathcal{I}[-1]$. Assume that ${\rm Ext}_{\mathcal{C}(Q)}^{1}(X, Y)=0$. If ${\rm Hom}_{\mathcal{C}(Q)}(X, Y)\neq 0$ and ${\rm Hom}_{\mathcal{C}(Q)}(Y, X)\neq 0$, then both $X$ and $Y$ are boundary objects.
\end{lemma}

\begin{proof}
Suppose to the contrary that $X$ is not a boundary object. Then $X$ lies in the $\tau_{\mathcal{C}}$-orbit of $P_{t}$, $t\geq 2$. Since ${\rm Hom}_{\mathcal{C}(Q)}(X, Y)\neq 0$ and ${\rm Hom}_{\mathcal{C}(Q)}(Y, X)\neq 0$, it follows that $D {\rm Ext}_{\mathcal{C}(Q)}^{1}(\tau_{\mathcal{C}}^{-1}Y, X)\neq 0$ and ${\rm Ext}_{\mathcal{C}(Q)}^{1}(\tau_{\mathcal{C}}Y, X)\neq 0$. Hence $\tau_{\mathcal{C}}^{-1}Y, \tau_{\mathcal{C}}Y\in H(X)\cap (\mathcal{P}\sqcup \mathcal{I}[-1])$.

Since $t\geq 2$, by Proposition \ref{coincide} it follows that both $\tau_{\mathcal{C}}^{-1}Y$ and $\tau_{\mathcal{C}}Y$ are in $H^{+}(X)\cup H^{-}(X)$. On the other hand, we have $D {\rm Ext}_{\mathcal{C}(Q)}^{1}(Y, X)=0$ and hence by Proposition \ref{coincide} again we have $Y\notin H^{+}(X)\cup H^{-}(X)$. If $\tau_{\mathcal{C}}Y\in H^{+}(X)$, then there exists paths $X\rightsquigarrow \tau_{\mathcal{C}}Y\rightsquigarrow Y$ in $\mathcal{P}\sqcup \mathcal{I}[-1]$ and thus $Y\in H^{+}(X)$, a contradiction. Thus $\tau_{\mathcal{C}}Y\notin H^{+}(X)$. Similarly we have $\tau_{\mathcal{C}}^{-1}Y\notin H^{-}(X)$.

Finally we have $\tau_{\mathcal{C}}Y\in H^{-}(X)$ and $\tau_{\mathcal{C}}^{-1}Y\in H^{+}(X)$. By definition there are two non-sectional paths $\tau_{\mathcal{C}}Y\rightsquigarrow X$ and $X\rightsquigarrow \tau_{\mathcal{C}}^{-1}Y$ in $\mathcal{P}\sqcup \mathcal{I}[-1]$. Since $\mathcal{P}\sqcup \mathcal{I}[-1]$ have neither projective object nor
injective object, it follows that we can find the paths $Y\rightsquigarrow X\rightsquigarrow Y$ in $\mathcal{P}\sqcup \mathcal{I}[-1]$, a contradiction. The proof of the lemma is completed.
\end{proof}

Finally we have the following result.

\begin{theorem}\label{Pisok}
Let $\mathcal{T}$ be a weakly cluster-tilting subcategory of $\mathcal{C}(Q)$. Let $X, Y\in \mathcal{T}$ be two objects lying in the connecting component $\mathcal{P}\sqcup \mathcal{I}[-1]$. Then for any morphisms $f: X\rightarrow Y$ and $g: Y\rightarrow X$ in $\mathcal{C}(Q)$ with $f,g\neq 0$, there exists $Z\in \mathcal{T}, Z\ncong X,Y$ such that either $f$ or $g$ factors through $Z$ in $\mathcal{C}(Q)$.
\end{theorem}

\begin{proof}
By Lemma \ref{force bo} it follows that both $X$ and $Y$ are boundary objects in $\mathcal{P}\sqcup \mathcal{I}[-1]$. Since ${\rm Hom}_{\mathcal{C}(Q)}(X, Y)\neq 0$ and ${\rm Hom}_{\mathcal{C}(Q)}(Y, X)\neq 0$, it is easy to see that there exists a path between $X$ and $Y$ in $\mathcal{P}\sqcup \mathcal{I}[-1]$. We may assume that this path starting in $X$ and ending in $Y$. 

We shall consider only the case where $X=\tau_{\mathcal{C}}^{2}P_{0}=A_{4}^{(0)}[-1]$. If $Y\in \mathcal{I}[-1]$, then $\tau_{\mathcal{C}}^{-2}Y\in \mathcal{P}$ and by Lemma \ref{uf}, ${\rm Hom}_{\mathcal{C}(Q)}(Y, X)\cong {\rm Hom}_{\mathcal{C}(Q)}(\tau_{\mathcal{C}}^{-2}Y, P_{0})\cong {\rm Hom}_{{\rm rep}(Q)}(\tau_{\mathcal{C}}^{-2}Y, P_{0})\oplus D{\rm Hom}_{D^{b}({\rm rep}(Q))}(P_{0}, Y)=0$, a contradiction. Hence $Y=A_{l}^{(k)}\in \mathcal{P}$, $l\geq 1$, $l$ is odd and $k\in \{0,1\}$. Now by Lemma \ref{uf} again, 
$$\begin{array}{rcl}
{\rm Hom}_{\mathcal{C}(Q)}(X, Y) & \cong & {\rm Hom}_{\mathcal{C}(Q)}(P_{0}, A_{l+4}^{(k)}) \\
                                                       & \cong & {\rm Hom}_{{\rm rep}(Q)}(P_{0}, A_{l+4}^{(k)})\oplus D{\rm Hom}_{D^{b}({\rm rep}(Q))}(A_{l+4}^{(k)}, \tau_{D}^{2}P_{0}) \\
                                                       & = & {\rm Hom}_{{\rm rep}(Q)}(P_{0}, A_{l+4}^{(k)})\neq 0
\end{array}$$
Thus, $k=1$ and ${\rm Hom}_{\mathcal{C}(Q)}(\tau_{\mathcal{C}}^{2}P_{0}, A_{l}^{(1)})\cong k$.

It is easy to see that there exists a unique non-boundary object $Z_{1}$ in $\mathcal{P}\sqcup \mathcal{I}[-1]$ such that $Z_{1}$ is a sectional successor of $\tau_{\mathcal{C}}^{2}P_{0}$ and a sectional predecessor of $A_{l}^{(1)}$, that is, there exists sectional paths $p: \tau_{\mathcal{C}}^{2}P_{0}\rightsquigarrow Z_{1}$ and $q: Z_{1}\rightsquigarrow A_{l}^{(1)}$ in $\mathcal{P}\sqcup \mathcal{I}[-1]$. Thus, $Z_{1}\ncong \tau_{\mathcal{C}}^{2}P_{0}, A_{l}^{(1)}$. Suppose first that $Z_{1}\in \mathcal{T}$. Since the connecting component $\mathcal{P}\sqcup \mathcal{I}[-1]$ of $\Gamma_{D^{b}({\rm rep}(Q))}$ is standard; see [\ref{LP1}], by definition there exists a $k$-equivalence 
$$G: k(\mathcal{P}\sqcup \mathcal{I}[-1])\rightarrow D^{b}({\rm rep}(Q))(\mathcal{P}\sqcup \mathcal{I}[-1]),$$ 
where $k(\mathcal{P}\sqcup \mathcal{I}[-1])$ is the mesh category and $D^{b}({\rm rep}(Q))(\mathcal{P}\sqcup \mathcal{I}[-1])$ is the full subcategory of $D^{b}({\rm rep}(Q))$ generated by the objects lying in $\mathcal{P}\sqcup \mathcal{I}[-1]$. It is not hard to see that 
$$qp\cong A_{4}^{(0)}[-1]\rightarrow B_{2,4}[-1]\rightarrow A_{2}^{(0)}[-1]\rightarrow A_{2,2}[-1]\rightarrow A_{1}^{(1)}\rightarrow B_{1,3}\rightarrow A_{3}^{(1)}\rightarrow \cdots\rightarrow A_{l}^{(1)}$$
is non-zero in $k(\mathcal{P}\sqcup \mathcal{I}[-1])$ and hence $G(qp)=G(q)\circ G(p)$ is non-zero in ${\rm Hom}_{D^{b}({\rm rep}(Q))}(\tau_{\mathcal{C}}^{2}P_{0}, A_{l}^{(1)})$. Since the canonical projection functor $\pi: D^{b}({\rm rep}(Q))\rightarrow \mathcal{C}(Q)$ is faithful, it follows that the composition $\xymatrix{\tau_{\mathcal{C}}^{2}P_{0}\ar[r]^{\pi(G(p))} & Z_{1}\ar[r]^{\pi(G(q))} & A_{l}^{(1)}}$ is non-zero in ${\rm Hom}_{\mathcal{C}(Q)}(\tau_{\mathcal{C}}^{2}P_{0}, A_{l}^{(1)})$. As a consequence, any morphism $f: \tau_{\mathcal{C}}^{2}P_{0}\rightarrow A_{l}^{(1)}$ factors through $Z_{1}$ in $\mathcal{C}(Q)$ because ${\rm Hom}_{\mathcal{C}(Q)}(\tau_{\mathcal{C}}^{2}P_{0}, A_{l}^{(1)})\cong k$. Our claim is established.

Suppose now that $Z_{1}\notin \mathcal{T}$. Then by definition there exists an indecomposable object $Z_{2}\in \mathcal{T}$ such that ${\rm Ext}_{\mathcal{C}(Q)}^{1}(Z_{2}, Z_{1})\neq 0$. If $Z_{2}\in \mathcal{R}$, then by Lemma \ref{uf} we have ${\rm Ext}_{\mathcal{C}(Q)}^{1}(Z_{2}, Z_{1})\cong D{\rm Hom}_{\mathcal{C}(Q)}(\tau_{\mathcal{C}}^{-1}Z_{1}, Z_{2})\cong D{\rm Hom}_{{\rm rep}(Q)}(\tau_{\mathcal{C}}^{-1}Z_{1}, Z_{2})\neq 0$. It is not hard to see that $\tau_{\mathcal{C}}^{-1}Z_{1}=A_{3,l+2}$ and hence by Proposition \ref{CtoR} we have $Z_{2}\in \mathcal{W}(A_{3,4})\cup \mathcal{W}(A_{l+1,l+2})$. On the other hand, since all three $\tau_{\mathcal{C}}^{2}P_{0}, A_{l}^{(1)}$ and $Z_{2}$ are in $\mathcal{T}$ and thus by Lemma \ref{uf} again we have 
$${\rm Ext}_{\mathcal{C}(Q)}^{1}(\tau_{\mathcal{C}}^{2}P_{0}, Z_{2})\cong {\rm Hom}_{\mathcal{C}(Q)}(P_{0}, \tau_{\mathcal{C}}^{-1}Z_{2})\cong {\rm Hom}_{{\rm rep}(Q)}(P_{0}, \tau_{Q}^{-1}Z_{2})=0,$$
$${\rm Ext}_{\mathcal{C}(Q)}^{1}(A_{l}^{(1)}, Z_{2})\cong {\rm Hom}_{\mathcal{C}(Q)}(A_{l}^{(1)}, \tau_{\mathcal{C}}Z_{2})\cong {\rm Hom}_{{\rm rep}(Q)}(P_{k'}, \tau_{Q}^{\frac{l+1}{2}}Z_{2})=0, k'\in \{0,1\}.$$
Applying Proposition \ref{CtoR} again the first equation implies that $Z_{2}\notin \mathcal{W}(A_{3,4})$ and the second equation implies that $Z_{2}\notin \mathcal{W}(A_{l+1,l+2})$, a contradiction.

Therefore, $Z_{2}\in H(Z_{1})\cap(\mathcal{P}\sqcup \mathcal{I}[-1])$. Since $Z_{1}$ is not a boundary object, by Proposition \ref{coincide} we obtain $Z_{2}\in H^{+}(Z_{1})\cup H^{-}(Z_{1})$. On the other hand, $Z_{2}\notin H(\tau_{\mathcal{C}}^{2}P_{0})$ and $Z_{2}\notin H(A_{l}^{(1)})$. Now by Proposition \ref{coincide} again it is not hard to see that $Z_{2}\in \{A_{l+2}^{(1)}, A_{l+4}^{(1)}, \cdots\}\cup \{A_{6}^{(0)}[-1], A_{8}^{(0)}[-1],\cdots\}$. In particular, $Z_{2}\ncong \tau_{\mathcal{C}}^{2}P_{0}, A_{l}^{(1)}$.

In the first case, let $Z_{2}=A_{l+j}^{(1)}, j>0, j~{\rm even}$. Then we define $f=(f_{i})_{i\geq 0}\in {\rm Hom}_{{\rm rep}(Q)}(A_{l}^{(1)}, A_{l+j}^{(1)})$ as $f_{0}=f_{2}=f_{3}=\cdots =f_{l}={\rm id}_{k}$ and $f_{1}=f_{s}=0, s\geq l+1$. Furthermore, we define $g=(g_{i})_{i\geq 0}\in {\rm Hom}_{{\rm rep}(Q)}(A_{l+j}^{(1)}, I_{0})$ as $g_{0}=g_{2}={\rm id}_{k}$ and $g_{1}=g_{s}=0, s\geq 3$. Since $(g\circ f)_{0}={\rm id}_{k}\neq 0$, it follows that $g\circ f\neq 0$ and hence $(g\circ f)_{D}=g_{D}\circ f_{D}$ is non-zero in ${\rm Hom}_{D^{b}({\rm rep}(Q))}(A_{l}^{(1)}, I_{0})$.

Let $F=\tau_{D}^{-1}\circ [1]$. Then it is easy to see that $F(\tau_{\mathcal{C}}^{2}P_{0})=I_{0}$. Now we define $u\in {\rm Hom}_{\mathcal{C}(Q)}(A_{l}^{(1)}, A_{l+j}^{(1)})$ as $u_{0}=f_{D}, u_{s}=0, s\neq 0$ and $v\in {\rm Hom}_{\mathcal{C}(Q)}(A_{l+j}^{(1)}, \tau_{\mathcal{C}}^{2}P_{0})$ as $v_{1}=g_{D}, v_{s}=0, s\neq 1$. Therefore, by definition, 
$$(v\circ u)_{1}=\bigoplus\limits_{p+q=1}(F^{p}(v_{q})\circ u_{p})=v_{1}\circ u_{0}=g_{D}\circ f_{D}\neq 0$$
and hence $\xymatrix{A_{l}^{(1)}\ar[r]^{u} & A_{l+j}^{(1)}\ar[r]^{v} & \tau_{\mathcal{C}}^{2}P_{0}}$ is non-zero in ${\rm Hom}_{\mathcal{C}(Q)}(A_{l}^{(1)}, \tau_{\mathcal{C}}^{2}P_{0})$. Finally by Lemma \ref{uf} we have 
$${\rm Hom}_{\mathcal{C}(Q)}(A_{l}^{(1)}, \tau_{\mathcal{C}}^{2}P_{0}) \cong {\rm Hom}_{\mathcal{C}(Q)}(A_{l}^{(1)}, I_{0}) \cong {\rm Hom}_{{\rm rep}(Q)}(A_{l}^{(1)}, I_{0}) \cong k.$$
Thus any morphism $A_{l}^{(1)}\rightarrow \tau_{\mathcal{C}}^{2}P_{0}$ factors through $A_{l+j}^{(1)}$ in $\mathcal{C}(Q)$. This establishes our claim.

In the second case, let $Z_{2}=A_{4+j}^{(0)}[-1], j>0, j~{\rm even}$. Then we define $f=(f_{i})_{i\geq 0}\in {\rm Hom}_{{\rm rep}(Q)}(A_{l}^{(1)}, \\A_{2+j}^{(1)})$ as $f_{0}=f_{2}=f_{3}=\cdots=f_{{\rm min}(l, 2+j)}={\rm id}_{k}$ and $f_{1}=f_{s}=0, s\geq {\rm min}(l, 2+j)+1$. Furthermore, we define $g=(g_{i})_{i\geq 0}\in {\rm Hom}_{{\rm rep}(Q)}(A_{2+j}^{(1)}, I_{0})$ as $g_{0}=g_{2}={\rm id}_{k}$ and $g_{1}=g_{s}=0, s\geq 3$. Since $(g\circ f)_{0}={\rm id}_{k}\neq 0$, it follows that $g\circ f\neq 0$ and hence $(g\circ f)_{D}=g_{D}\circ f_{D}$ is non-zero in ${\rm Hom}_{D^{b}({\rm rep}(Q))}(A_{l}^{(1)}, I_{0})$.

It is easy to see that $F(A_{4+j}^{(0)}[-1])=A_{2+j}^{(1)}$. Now we define $u\in {\rm Hom}_{\mathcal{C}(Q)}(A_{l}^{(1)}, A_{4+j}^{(0)}[-1])$ as $u_{1}=f_{D}, u_{s}=0, s\neq 1$ and $v\in {\rm Hom}_{\mathcal{C}(Q)}(A_{4+j}^{(0)}[-1], \tau_{\mathcal{C}}^{2}P_{0})$ as $v_{0}=F^{-1}(g_{D}), v_{s}=0, s\neq 0$. Therefore, by definition, 
$$(v\circ u)_{1}=\bigoplus\limits_{p+q=1}(F^{p}(v_{q})\circ u_{p})=F(v_{0})\circ u_{1}=g_{D}\circ f_{D}\neq 0$$
and hence $\xymatrix{A_{l}^{(1)}\ar[r]^{u} & A_{4+j}^{(0)}[-1]\ar[r]^{v} & \tau_{\mathcal{C}}^{2}P_{0}}$ is non-zero in ${\rm Hom}_{\mathcal{C}(Q)}(A_{l}^{(1)}, \tau_{\mathcal{C}}^{2}P_{0})$. Since ${\rm Hom}_{\mathcal{C}(Q)}(A_{l}^{(1)}, \\\tau_{\mathcal{C}}^{2}P_{0})\cong k$, it follows that any morphism $A_{l}^{(1)}\rightarrow \tau_{\mathcal{C}}^{2}P_{0}$ factors through $A_{4+j}^{(0)}[-1]$ in $\mathcal{C}(Q)$. The proof of the theorem is completed.
\end{proof}

We are ready to obtain our main result of this paper.

\begin{theorem}\label{cluster st}
The category $\mathcal{C}(Q)$ is a cluster category.
\end{theorem}

\begin{proof}
By Remark \ref{CQ}, $\mathcal{C}(Q)$ is a Hom-finite Krull-Schmidt $2$-Calabi-Yau triangulated $k$-category. Since $Q$ is a locally finite quiver with no infinite path, according to [\ref{LP2}, (4.4)] $\mathcal{C}(Q)$ has a cluster-tilting subcategory and hence we need only to show that the quiver $Q_{\mathcal{T}}$ of every cluster-tilting subcategory $\mathcal{T}$ of $\mathcal{C}(Q)$ has no oriented cycle of length one or two; see [\ref{BIRS}, (I.1.6)].

Let $X\in \mathcal{T}$. Since $Q$ is an infinite Dynkin quiver with no infinite path, according to [\ref{LP2}, (4.6)] ${\rm End}_{\mathcal{C}(Q)}(X)\cong k$ and hence each non-zero morphism $X\rightarrow X$ is an isomorphism and so not irreducible. Therefore, there are no oriented cycle of length one in $Q_{\mathcal{T}}$.

Let $X, Y\in \mathcal{T}$. Suppose to the contrary that there is an oriented cycle $\xymatrix{X\ar[r]^{f} & Y\ar[r]^{g} & X}$ in $Q_{\mathcal{T}}$. In particular, ${\rm Hom}_{\mathcal{C}(Q)}(X, Y)\neq 0$ and ${\rm Hom}_{\mathcal{C}(Q)}(Y, X)\neq 0$. Thus by Lemma \ref{Rok} it follows that at least one of $X$ and $Y$ is not regular. Suppose first that $X$ is lying in the connecting component $\mathcal{P}\sqcup \mathcal{I}[-1]$ and $Y$ is lying in the regular component $\mathcal{R}$. Then by Corollary \ref{PRok} we obtain that there exists $Z\in \mathcal{T}, Z\ncong X,Y$ such that either $f$ or $g$ factors through $Z$ in $\mathcal{C}(Q)$. Since $\mathcal{T}$ is a full subcategory of $\mathcal{C}(Q)$, it follows that either $f$ or $g$ is not irreducible, a contradiction. Suppose now that both $X$ and $Y$ are lying in the connecting component $\mathcal{P}\sqcup \mathcal{I}[-1]$. Then by Theorem \ref{Pisok} we can obtain a similar contradiction. Therefore, there are no oriented cycle of length two in $Q_{\mathcal{T}}$. The proof of the theorem is completed.
\end{proof}

\begin{remark}\label{ar D}
Let $Q'$ be any connected infinite Dynkin quiver of type $D_{\infty}$ with no infinite path. Then Reiten and Van den Bergh construct a category $\widetilde{{\rm rep}}(Q')$ in {\rm [\ref{RV}]}. Since $Q$ is a section in $\mathbb{Z}Q'$, they show that $\widetilde{{\rm rep}}(Q')\cong {\rm rep}(Q')$ and the categories $\widetilde{{\rm rep}}(Q)$ and $\widetilde{{\rm rep}}(Q')$ are derived equivalent. As a consequence, the categories ${\rm rep}(Q)$ and ${\rm rep}(Q')$ are also derived equivalent. Therefore, since the category $\mathcal{C}(Q)$ is a cluster category, the category $\mathcal{C}(Q')$ is also a cluster category for any connected infinite Dynkin quiver $Q'$ of type $D_{\infty}$ with no infinite path.
\end{remark}

\bigskip

{\bf Acknowledgements.}~~ This work was carried out when the author is a postdoctoral fellow at Universit\'{e} de Sherbrooke. The author thanks Prof. Shiping Liu for his helpful discussions and warm hospitality. He also wants to thank people in Universit\'{e} de Sherbrooke for their help.

\end{document}